\documentclass[11pt]{amsart}
\usepackage[leqno]{amsmath}
\usepackage{amssymb}
\usepackage{bbm}
\usepackage{enumerate}
\usepackage{varwidth}
\usepackage[dvipsnames]{xcolor}
\usepackage{comment}
\usepackage{subfigure}
\usepackage[
colorlinks=true,linkcolor=NavyBlue,urlcolor=RoyalBlue,citecolor=PineGreen,bookmarks=true,bookmarksopen=true,bookmarksopenlevel=2,unicode=true,linktocpage]{hyperref}
\usepackage{microtype}
\usepackage{mathrsfs}
\usepackage{graphicx}
\usepackage[all]{xy}
\numberwithin{equation}{section}

\newtheorem{theorem}{Theorem}[section]

\newtheorem{lemma}[theorem]{Lemma}
\newtheorem{proposition}[theorem]{Proposition}

\newcommand{\eqb}{\begin{equation}}
\newcommand{\eqe}{\end{equation}}

\renewcommand{\C}{\mathbf{C}}
\newcommand{\D}{\mathbf{D}}
\newcommand{\E}{\mathbf{E}}

\newcommand{\h}{\mathbf{H}}
\newcommand{\N}{\mathbf{N}}
\newcommand{\Z}{\mathbf{Z}}
\newcommand{\p}{\mathbf{P}}

\newcommand{\R}{\mathbf{R}}

\newcommand{\Fh}{\mathfrak {h}}

\newcommand{\mcl}{\mathcal}

\newcommand{\CF}{\mathcal {F}}

\newcommand{\CZ}{\mathcal {Z}}

\newcommand{\SLE}{{\rm SLE}}

\newcommand{\dist}{\mathrm{dist}}

\newcommand{\diam}{\mathrm{diam}}

\newcommand{\im}{\mathrm{Im}}
\newcommand{\re}{\mathrm{Re}}

\newcommand{\hull}{\mathrm{hull}}

\newcommand{\one}{{\bf 1}}

\newcommand{\wt}{\widetilde}
\newcommand{\ol}{\overline}
\newcommand{\op}{\operatorname}
\newcommand{\ul}{\underline}

\renewcommand{\emph}[1]{{\it #1}}

\newcommand{\giv}{\,|\,}

\newcommand{\vareps}{\varepsilon}

\newcommand{\ep}{\epsilon}

\newcommand{\rta}{\rightarrow}

\newcommand{\ve}{\vareps}

\newcommand{\sol}[1]{{}}

\newcommand{\wh}{\widehat}

\newcommand{\hcap}{{\mathrm{hcap}}}

\newcommand{\neigh}[2]{{\mathfrak N}_{#1}(#2)}
\newcommand{\ineigh}[2]{#2_{#1}}
\newcommand{\db}{\mathrm d }
\newcommand{\rev}[1]{\ol{#1}}

\newcommand{\closure}[1]{\mathrm{cl}(#1)}

\begin{document}

\title{A continuous proof of the existence of the $\SLE_8$ curve}
\author{Valeria Ambrosio and Jason Miller}

\begin{abstract}
Suppose that $\eta$ is a whole-plane space-filling $\SLE_\kappa$ for $\kappa \in (4,8)$ from $\infty$ to $\infty$ parameterized by Lebesgue measure and normalized so that $\eta(0) = 0$. For each $T > 0$ and $\kappa \in (4,8)$ we let $\mu_{\kappa,T}$ denote the law of $\eta|_{[0,T]}$. We show for each $\nu, T > 0$ that the family of laws $\mu_{\kappa,T}$ for $\kappa \in [4+\nu,8)$ is compact in the weak topology associated with the space of probability measures on continuous curves $[0,T] \to \C$ equipped with the uniform distance. As a direct byproduct of this tightness result (taking a limit as $\kappa \uparrow 8$), we obtain a new proof of the existence of the $\SLE_8$ curve which does not build on the discrete uniform spanning tree scaling limit of Lawler-Schramm-Werner.
\end{abstract}

\maketitle

\setcounter{tocdepth}{1}
\tableofcontents

\parindent 0 pt
\setlength{\parskip}{0.20cm plus1mm minus1mm}

\section{Introduction}

The chordal Schramm-Loewner evolution ($\SLE_\kappa)$ is a one parameter family of random curves indexed by $\kappa > 0$ which connect two distinct boundary points in a simply connected domain $D$. It was introduced by Schramm~\cite{s2000sle} as a candidate for the scaling limit of the interfaces in discrete models from statistical mechanics on planar lattices at criticality. It has since been shown to arise as such a scaling limit in a number of different cases, e.g., ~\cite{s2001cardy,lsw2004ust,ss2009contour,s2010ising}. $\SLE_\kappa$'s are also important in the context of Liouville quantum gravity (LQG)~\cite{ds2011kpz,s2016zipper,dms2021mating} and have been shown in this setting to describe the scaling limit of the interfaces for discrete models on random planar maps, e.g., \cite{she2016hamburger,kmsw2019bipolar,gm2021saw,gm2017percolation,lsw2017schnyder}.

$\SLE_\kappa$ is typically defined in the upper half-plane $\h$ going from $0$ to $\infty$. The starting point for the definition is the \emph{chordal Loewner equation} which for a continuous function $W \colon \R_+ \to \R$ is given by
\begin{equation}
\label{eqn:loewner_ode}
\partial_t g_t(z) = \frac{2}{g_t(z) - W_t}, \quad g_0(z) = z.
\end{equation}
For each $z \in \h$, we have that $t \mapsto g_t(z)$ is defined up to $T_z = \inf\{t \geq 0 : \im(g_t(z)) = 0\}$. Let $K_t = \{z \in \h : t \leq T_z\}$. Then $g_t$ is the unique conformal transformation $\h \setminus K_t \to \h$ normalized to look like the identity at $\infty$, i.e., $g_t(z) - z \to 0$ as $z \to \infty$. The family of conformal maps $(g_t)$ is called the \emph{Loewner chain} associated with the function $W$. For a set $X$ we denote by $\closure{X}$ its closure. We say that the Loewner chain $(g_t)$ is generated by the continuous curve $\eta \colon \R_+ \to \closure{\h}$ if for each $t \geq 0$ we have that $\h \setminus K_t$ is equal to the unbounded component of $\h \setminus \eta([0,t])$. We note that it is not true that every Loewner chain driven by a continuous function~$W$ is generated by a continuous curve~\cite[Example~4.27]{lawler2005conformally}, however it was proved by Marshall and Rohde~\cite{mr2005loewner} that there exists a constant $C > 0$ so that if the $1/2$-H\"older norm of $W$ is at most $C$ then the Loewner chain is generated by a continuous curve.

Schramm made the crucial observation in~\cite{s2000sle} that if one encodes the scaling limit of a discrete lattice model in terms of the Loewner flow~\eqref{eqn:loewner_ode} and assumes that it is conformally invariant and satisfies a version of the spatial Markov property then the driving function $W$ must have stationary and independent increments so that $W = \sqrt{\kappa} B$ where $B$ is a standard Brownian motion and $\kappa > 0$. This led him to define $\SLE_\kappa$ to be the Loewner chain driven by this choice of $W$. Since Brownian motion is not $1/2$-H\"older continuous, it is not obvious that with this choice of $W$ the associated Loewner flow $(g_t)$ is generated by a continuous curve. That this is the case was proved for $\kappa \neq 8$ by Rohde and Schramm in~\cite{rs2005basicproperties} and for $\kappa = 8$ by Lawler, Schramm, and Werner in~\cite{lsw2004ust}. The proof given in~\cite{rs2005basicproperties} is based on estimating the moments of $(g_t^{-1})'$ while the proof for $\kappa =8$ in~\cite{lsw2004ust} is based on the convergence of the uniform spanning tree (UST) on a subgraph of $\Z^2$ to $\SLE_8$ and hence does not make use of ``continuous methods''.

The reason that the proof given in~\cite{rs2005basicproperties} is restricted to the case that $\kappa \neq 8$ is related to the regularity of $\SLE_8$. In particular, it turns out that the $\SLE_\kappa$ curves are H\"older continuous for $\kappa \neq 8$~\cite{jvl2011optimal,l2008holder} but have modulus of continuity $(\log \delta^{-1})^{-1/4+o(1)}$ as $\delta \to 0$ for $\kappa = 8$~\cite{kms2021regularity}, which was previously conjectured by Alvisio and Lawler~\cite{ml2014sle8}.

Recall that if $K \subseteq \h$ is such that $\closure{K}$ is compact and $\h \setminus K$ is simply connected then the half-plane capacity of $K$ is defined by
\begin{equation}
\label{eqn:hcap_def}
\hcap(K) = \lim_{y \to \infty} y \E_{iy}[ \im(B_\tau)] 
\end{equation}
where $\E_z$ denotes the expectation with respect to the law of a complex Brownian motion starting from $z$ and $\tau = \inf\{t \geq 0 : B_t \notin \h \setminus K\}$. Implicit in~\eqref{eqn:loewner_ode} is that time is parameterized by (half-plane) capacity which means that $\hcap(K_t) = 2t$ for all $t \geq 0$.

We will now explain why the deterioration of the regularity of the $\SLE_\kappa$ curves as $\kappa \to 8$ is closely related to the choice of the capacity time parameterization. The $\SLE_\kappa$ curves turn out to be simple for $\kappa \leq 4$, self-intersecting but not space-filling for $\kappa \in (4,8)$, and space-filling for $\kappa \geq 8$~\cite{rs2005basicproperties}. In particular, the value $\kappa = 8$ is special because it is the critical value at or above which $\SLE_\kappa$ is space-filling. This is reflected in the fact that the left and right parts of the outer boundary of an $\SLE_8$ come very close to intersecting each other (but do not for other values of $\kappa$). A more precise version of this statement is as follows. Suppose that $\eta$ is an $\SLE_8$ in $\h$ from $0$ to $\infty$ and $t \in [0,1]$. Fix $\epsilon > 0$ and let $\tau = \inf\{s \geq t : |\eta(s) - \eta(t)| = \epsilon\}$. In view of~\eqref{eqn:hcap_def} we have that the amount of capacity time elapsed between the times $t$ and $\tau$, i.e., $\tau-t$, is related to the harmonic measure of $B(\eta(t),\epsilon)$ in $\h \setminus \eta([0,t])$. It was shown in~\cite{kms2021regularity} that this harmonic measure can decay as fast as $\exp(-\epsilon^{-4+o(1)})$ as $\epsilon \to 0$ which is in contrast to the case $\kappa \neq 8$ where it can only decay as fast as a power of $\epsilon$~\cite{jvl2012tipspectrum}.

In order to circumvent this issue, one can consider other time parameterizations of $\SLE_\kappa$ which are not based on harmonic measure. One important example is the so-called \emph{natural parameterization}~\cite{ls2011natural,lr2015minkowski} and in this case one has different regularity for the $\SLE_\kappa$ curves~\cite{z2019optimal} which is H\"older up to and including at $\kappa = 8$. In the case that $\kappa \geq 8$ so that $\SLE_\kappa$ is space-filling the natural parameterization is equivalent to parameterizing the curve according to (planar) Lebesgue measure.  By this, we mean that the Lebesgue measure of $\eta([0,t])$ is equal to $t$ for each $t \geq 0$.  The fact that $\SLE_\kappa$ for $\kappa \geq 8$ with the Lebesgue measure parameterization is locally $\alpha$-H\"older continuous for every $\alpha \in (0,1/2)$ was proved in~\cite{ghm2020kpz}. The results of~\cite{ghm2020kpz} in fact apply to space-filling $\SLE_\kappa$ for $\kappa > 4$ \cite{ms2016ig1,ms2017ig4}. Space-filling $\SLE_\kappa$ agrees with ordinary $\SLE_\kappa$ for $\kappa \geq 8$ and, roughly, for $\kappa \in (4,8)$ it evolves as an ordinary $\SLE_\kappa$ except whenever it disconnects a region it fills it in with a space-filling loop. We will review the definition in Section~\ref{subsec:space_filling_sle}. In order to simplify our arguments, we will focus on the case $\kappa \in (4,8)$ and we will work in the whole-plane. Our first main result is the compactness of the law of whole-plane space-filling $\SLE_\kappa$ as $\kappa$ varies.

\begin{theorem}
\label{thm:compactness}
For each $T > 0$ and $\kappa \in (4,8)$ we let $\mu_{\kappa,T}$ denote the law of $\eta|_{[0,T]}$ where $\eta$ is a whole-plane space-filling $\SLE_\kappa$ from $\infty$ to $\infty$ parameterized by Lebesgue measure and normalized so that $\eta(0) = 0$. For each $\nu, T > 0$ the family of laws $\mu_{\kappa,T}$ for $\kappa \in [4+\nu,8)$ is compact in the weak topology associated with the space of probability measures on continuous curves $[0,T] \to \C$ equipped with the uniform distance.
\end{theorem}

The reason that we need $\kappa$ bounded away from $4$ in the statement of Theorem~\ref{thm:compactness} is that  space-filling $\SLE_\kappa$ is only defined for $\kappa > 4$ and the regularity of space-filling $\SLE_\kappa$ deteriorates as $\kappa \downarrow 4$. In the present paper, we will show that the estimates of~\cite{ghm2020kpz} can be made quantitative in $\kappa$ and use this to prove Theorem~\ref{thm:compactness}. From the compactness result from Theorem~\ref{thm:compactness}, we can take a subsequential limit as $\kappa \uparrow 8$. We will show that after conformally mapping to $\h$ and reparameterizing by half-plane capacity the resulting curve has a continuous Loewner driving function which is given by $\sqrt{8}B$, $B$ a standard Brownian motion, hence is $\SLE_8$. This leads to our second theorem.

\begin{theorem}
\label{thm:sle8_continuous}
Chordal $\SLE_8$ is a.s.\ generated by a continuous curve.	
\end{theorem}

\subsection*{Outline}

The remainder of this article is structured as follows. In Section~\ref{sec:preliminaries}, we will collect some preliminaries on $\SLE$ as well as on its relationship with the Gaussian free field (GFF). We will then establish uniform estimates for GFF flow lines in Section~\ref{sec:uniform_flow_line_behavior} and then use these estimates in Section~\ref{sec:ball_pocket} to obtain a uniform estimate for the probability that a space-filling $\SLE$ fills a ball. We will then deduce the tightness of the law of space-filling $\SLE$ as $\kappa$ varies in Section~\ref{sec:tightness}  (giving Theorem~\ref{thm:compactness}) and finally show in Section~\ref{sec:uniqueness_of_limit} that the limit as $\kappa \uparrow 8$ is an $\SLE_8$ (giving Theorem~\ref{thm:sle8_continuous}).  We stress that the tools used in this paper are based on the imaginary geometry approach to $\SLE_\kappa$ \cite{ms2016ig1,ms2017ig4} and do not make use of LQG.

\subsection*{Acknowledgements} V.A.\ and J.M.\ were supported by the ERC starting grant 804166 (SPRS).  We thank Wendelin Werner for helpful comments on an earlier version of this article.

\section{Preliminaries}
\label{sec:preliminaries}

\subsection{Schramm-Loewner evolutions}
\label{subsec:sle}

As we mentioned above, $\SLE_\kappa$ is the random family of hulls $(K_t)$ which arise from the solution to~\eqref{eqn:loewner_ode} in the case that $W = \sqrt{\kappa} B$ and $B$ is a standard Brownian motion. In this work, it will be important for us also to consider the so-called $\SLE_\kappa(\rho)$ processes~\cite[Section~8.3]{lsw2003confres}, which are a variant of $\SLE_\kappa$ where one keeps track of extra marked points. The $\SLE_\kappa(\ul{\rho})$ process with weights $\rho_1,\ldots,\rho_n \in \R$ located at the force points $x_1,\ldots,x_n \in \closure{\h}$ is defined by solving~\eqref{eqn:loewner_ode} with $W$ taken to be the solution to the SDE
\begin{align}
\label{eqn:sle_kappa_rho_equation}
dW_t = \sqrt{\kappa} dB_t + \sum_{j=1}^n \re\left( \frac{\rho_j}{W_t-V_t^j} \right) dt, \quad dV_t^j = \frac{2}{V_t^j - W_t} dt,\quad V_0^j = x_j\quad\text{for}\quad j=1,\ldots,n.
\end{align}
It was shown in~\cite{ms2016ig1} that this has a solution up until the so-called \emph{continuation threshold} which is defined as
\[ \inf\left\{ t \geq 0 : \sum_{j=1}^n \rho_j \one_{\{W_t = V_t^j \}} \leq -2 \right\}.\]
Moreover, it is proved in~\cite{ms2016ig1} that the associated process is generated by a continuous curve (up to the continuation threshold). For single force point $\SLE_\kappa(\rho)$ processes, one can construct the solution to the SDE~\eqref{eqn:sle_kappa_rho_equation} for the driving function directly using Bessel processes.

Recall that a solution to the SDE
\begin{equation}
\label{eqn:bessel_sde}
 dX_t = \frac{\delta-1}{2} \cdot \frac{1}{X_t} dt + dB_t
\end{equation}
is a Bessel process of dimension $\delta$. Fix $\rho > -2$, let
\begin{equation}
\label{eqn:bessel_rho_formula}
 \db=\db(\kappa,\rho) = 1 + \frac{2(\rho+2)}{\kappa},
\end{equation}
and suppose that $X$ is a Bessel process of dimension $\db$. Set
\[ W_t = V_t - \sqrt{\kappa} X_t \quad\text{and}\quad V_t = \frac{2}{\sqrt{\kappa}} \int_0^t \frac{1}{X_s} ds + \sqrt{\kappa} X_0.\]
Then $(W,V)$ has the law of the driving function of an $\SLE_\kappa(\rho)$ process with force point located at $V_0 = \sqrt{\kappa} X_0 \geq 0$. It will not be important for this work, but one can use this to define single force point $\SLE_\kappa(\rho)$ processes for $\rho < -2$ and the continuity of these processes was proved in~\cite{msw2017cleperc,ms2019lightcone}.

In this work, we will also need to consider radial and whole-plane $\SLE_\kappa(\rho)$ processes. The radial Loewner equation is given by
\[ \partial_t g_t(z) = g_t(z) \frac{W_t + g_t(z)}{W_t - g_t(z)},\quad g_0(z) = z.\]
Ordinary radial or whole-plane $\SLE_\kappa$ is given by taking $W_t = e^{i\sqrt{\kappa} B_t}$ where $B$ is a standard Brownian motion. Let
\[ \Psi(z,w) = -z \frac{z+w}{z-w} \quad\text{and}\quad \wt{\Psi}(z,w) = \frac{\Psi(z,w) + \Psi(1/\ol{z},w)}{2}.\]
Suppose that $(W,O)$ solves the equation
\[ dW_t = \left(\frac{\rho}{2} \wt{\Psi}(O_t,W_t) -\frac{\kappa}{2} W_t\right) dt + i \sqrt{\kappa} W_t dB_t \quad\text{and}\quad d O_t = \Psi(W_t,O_t) dt.\]
Then we obtain a radial or whole-plane $\SLE_\kappa(\rho)$ by solving the radial Loewner equation with this choice of driving function.

Let $\theta_t = \arg W_t - \arg O_t$. Then we have that
\[ d\theta_t = \frac{\rho+2}{2} \cot(\theta_t/2) dt + \sqrt{\kappa} dB_t\]
where $B$ is a standard Brownian motion. A \emph{radial Bessel process} solves the SDE
\begin{equation}
\label{eqn:radial_bessel}
d X_t = \frac{\delta-1}{4} \cot(X_t/2) dt + dB_t
\end{equation}
where $B$ is a standard Brownian motion. Note that $\theta_{t/\kappa}$ is a radial Bessel process of dimension $\delta(\kappa,\rho) = 1+2(\rho+2)/\kappa$. Locally near $0$ and $2\pi$ a radial Bessel process behaves like an ordinary Bessel process of the same dimension.

\subsection{Continuity of Bessel processes}\label{subsec:bessel}

In what follows later, we will need that the law of an $\SLE_\kappa(\rho)$ process is continuous in $\kappa,\rho$. This is a consequence of the continuity of a Bessel process w.r.t.\ its dimension, as we will explain in the following lemma. We recall that one can construct a Bessel process of dimension $\db > 0$ as the square root of the so-called square Bessel process \cite{ry1999martingales}, which is the solution to the SDE
\[ dZ_t = 2 \sqrt{Z_t} dB_t + \db\times dt.\]

\begin{lemma}
\label{lem:bes-cont}
Let $(\db_n)$ be a sequence in $\R$ converging to some $ \db_*\in\R$ and let $ \db_-=\inf\{ \db_n\}$. Suppose we couple Bessel processes $Z^{ \db_n}$ for $n \in \N$, $Z^{ \db_*}$, $Z^{ \db_-}$ of dimension $ \db_n$, $ \db_*$, and $ \db_-$, respectively, with the same starting point on a common probability space using the same Brownian motion $B$. Then, for $T>0$, on the event that $Z^{ \db_-}|_{[0,T]}$ does not hit the origin, $Z^{ \db_n}|_{[0,T]} \to Z^{ \db_*}|_{[0,T]}$ uniformly. Moreover, if $ \db_->1$ then a.s.\ $Z^{ \db_n} \to Z^{ \db_*}$ locally uniformly.
 \end{lemma}

\begin{proof}

We begin with the second part of the statement. For a Bessel process $Z^ \db_t$ of dimension $ \db$, let $X^ \db_t:=(Z_t^ \db)^2$ be the corresponding squared Bessel process. That is
\[
Z^ \db_t= Z^ \db_0+\frac{ \db-1}{2}\int_0^t \frac{1}{Z^ \db_s}ds+B_t \quad\text{and}\quad X^ \db_t= X^ \db_0+{2}\int {\sqrt X^ \db_s}dB_s+ \db\times t.
\]

Assume that $Z^{ \db_n}$ for $n \in \N$, $Z^{ \db_*}$, $Z^{ \db_-}$ are coupled together on a common probability space using the same Brownian motion $B$ and assume that the starting points are such that $Z_0^{ \db_-}\leq Z_0^{ \db_n}$ for every $n \in \N$. By \cite[ Chapter 9, Theorem~3.7]{ry1999martingales} it a.s.\ holds that $X_t^ \db\leq X_t^{ \db'}$ for all $t \geq0$ provided $ \db\leq \db'$ and $X_0^{ \db}\leq X_0^{ \db'}$. By taking the square root, we see that $Z_t^ \db\leq Z_t^{ \db'}$ a.s.\ for all $t \geq 0$.

Let $ \db_->1$ be fixed and $ \db^*, \db_n\in[ \db_-,\infty)$. Fix $T>0$ and suppose that $\db^*> \db_n$ (which implies $Z_t^{ \db_-}\leq Z_t^{ \db_n}\leq Z_t^{ \db^*}$). Assuming $Z_0^{\db_-} = Z_0^{\db_n} = Z_0^{\db^*}$ we have that
\begin{align*}
0\leq Z^{ \db^*}_t-Z^ { \db_n}_t&= \frac{ \db^*-1}{2}\int_0^t \frac{1}{Z^{ \db^*}_s}ds- \frac{ { \db_n}-1}{2}\int_0^t \frac{1}{Z^{ \db_n}_s}ds \\
&=\frac{ \db^*-1}{2}\int_0^t \left(\frac{1}{Z^{ \db^*}_s}- \frac{1}{Z^{ \db_n}_s}\right) ds+\frac{ \db^*- { \db_n}}{2}\int_0^t\frac{1}{Z^{ \db_n}_s}ds\\
&\leq \frac{ \db^*- { \db_n}}{2}\int_0^T\frac{1}{Z^ { \db_n}_s}ds\leq \frac{ \db^*- { \db_n}}{2}\int_0^T\frac{1}{Z^{ \db_-}_s}ds.
\end{align*}
By applying the above by swapping the roles of $ \db_n, \db^*\in[ \db_-,\infty)$ if necessary, we have that
\[ \sup_{t\in[0,T]}| Z^{ \db^*}_t-Z^{ \db_n}_t|\leq \frac{| \db^*- \db_n|}{2}\int_0^T\frac{1}{Z^{ \db_-}_s}ds.\]
Since $ \db_->1$, the integral on the right hand side is a.s.\ finite. This completes the proof of the second part of the statement.

For the first part of the statement we have that $ \db_-\leq \db\leq \db'$ implies $Z^{ \db_-}\leq Z^ \db\leq Z^{ \db'}$. Assuming we are on the event $u=\inf_{t\in[0,T]}Z^{ \db_-}_t>0$ we have that
\begin{align*}
0\leq Z^{ \db'}_t-Z^ \db_t
&=\frac{ \db'-1}{2}\int_0^t \left(\frac{1}{Z^{ \db'}_s}- \frac{1}{Z^ \db_s}\right) ds+\frac{ \db'- \db}{2}\int_0^t\frac{1}{Z^ \db_s}ds\\
&\leq\frac{| \db'-1|}{2u^2}\int_0^t \left({Z^{ \db'}_s-Z^{ \db}_s}\right) ds+\frac{ \db'- \db}{2}\frac{t}{u}.\\
\end{align*}
 Gronwall's lemma implies
\[0\leq \sup_{t\in[0,T]} (Z^{ \db'}_t-Z^ \db_t)\leq \exp\left({T\frac{|1- \db'|}{2u^2}}\right)\frac{ \db'- \db}{2}\frac{T}{u},\]
and the first part of the statement follows.
\end{proof}

Since the Loewner driving function of an $\SLE_\kappa(\rho)$ process can be constructed by a Bessel process of dimension $\db = 1+ 2(\rho+2)/\kappa$ (recall~\eqref{eqn:bessel_rho_formula}) it follows from Lemma~\ref{lem:bes-cont} and~\cite[Proposition~4.43]{lawler2005conformally} that one has the Carath\'eodory convergence of chordal $\SLE_{\kappa_n}(\rho_n)$ run up to any fixed and finite time to a $\SLE_{\kappa_0}(\rho_0)$ as $(\kappa_n, \rho_n)\to (\kappa_0,\rho_0)$.

\subsection{Imaginary geometry review}
\label{subsec:gff}

We will now review some of the basics from~\cite{ms2016ig1,ms2017ig4} which will be important for this article.

\subsubsection{Boundary data}

We briefly recall how the coupling between the coupling between an $\SLE_\kappa (\rho)$ curve and a GFF with suitable boundary condition works. We start by setting some notation. For $\kappa \in (0,4)$ and $\kappa'=16/\kappa$ let
\begin{equation}
 \lambda = \frac{\pi}{\sqrt{\kappa}},\quad \lambda' = \frac{\pi}{\sqrt{\kappa'}} , \quad \text{and}\quad \chi = \frac{2}{\sqrt{\kappa}} - \frac{\sqrt{\kappa}}{2}.
\end{equation}

Let $\eta$ be an $\SLE_\kappa(\ul{\rho})$ process from $0$ to $\infty$ in $\h$ with
\[ \ul{\rho} = (\ul{\rho}^L;\ul{\rho}^R) = (( \rho_{k,L}, \ldots, \rho_{1,L}), (\rho_{1,R}, \ldots, \rho_{\ell,R})) \] 
with force points located at $\ul{x} = {(\ul{x}^L ; \ul{x}^R)} = (x_{k,L} < \cdots < x_{1,L} < 0 < x_{1,R} < \cdots < x_{\ell,R})$. Let
\[ \ol{\rho}_{j,L} = \sum_{i=1}^j \rho_{i,L} \quad\text{for}\quad 1 \leq j \leq k \quad\text{and}\quad \ol{\rho}_{j,R} = \sum_{i=1}^j \rho_{i,R} \quad\text{for}\quad 1 \leq j \leq \ell\]
where we take the convention $\overline \rho_{0,L} = \overline \rho_{0,R} = 0$.

The boundary conditions for the GFF coupled with $\eta$ are as follows. Let $\Fh_0$ denote the bounded harmonic function in $\h$ with boundary conditions
\begin{equation}
 -\lambda ( 1+ \overline{\rho}_{j, L}) \quad\text{for}\quad x \in (x_{j+1,L},x_{j,L}] \quad\text{and}\quad
 \lambda (1+ \overline \rho_{i,R}) \quad\text{for}\quad x \in (x_{i,R},x_{i+1,R}]
\end{equation}
for all $0 \leq j \leq k $ and $0 \leq i \leq \ell$, with $x_{0,L} = x_{0,R} = 0$, and $x_{k+1,L} = - \infty$, $x_{\ell+1,R} = \infty$.

For each $t \geq 0$, let $\Fh_t$ be the harmonic function in the complement of the curve $\eta([0,t])$ with boundary data that can be informally defined as follows: at a given time $s$ the boundary data on the left side of the curve is given by
$-\lambda' + \chi \cdot {\rm winding}$, where the winding is the winding of $\eta([0,s])$. On the right side of the curve, the boundary data is $\lambda' + \chi \cdot {\rm winding}$, and on $\partial \h$, one uses the same boundary data as $\Fh_0$. That is, we can set $\Fh_t = \Fh_0 \circ f_t - \chi \arg f_t'$ where $f_t = g_t - W_t$ and~$g_t$ is the unique conformal map from the unbounded component of $\h \setminus \eta([0,t])$ to~$\h$ with $g_t(z) - z \to 0$ as $z \to \infty$. 
Then, it holds for any stopping time $\tau$ the curve $\eta$ up to $\tau$ defines a local set of a GFF $h$ in $\h$ with boundary conditions $\Fh_0$ (\cite[Theorem~1.1]{ms2016ig1})

As shown in~\cite[Theorem~1.2]{ms2016ig1}, in this coupling the curve $\eta$ can then be deterministically recovered from the GFF and it is referred to as the \emph{flow line} starting at $0$ and targeted at $\infty$ of the GFF $h$ with the boundary conditions $\Fh_0$. 

One can then define the flow line with angle $\theta$ associated with a GFF $h$ (with certain boundary conditions), to be the flow line of $h + \theta \chi$. 

The same construction works for an $\SLE_{\kappa'}(\ul{\rho}')$ process starting from $0$, up to a sign change (to accommodate for the $\chi = - \chi'$ change). The boundary conditions in this case are given by:
\begin{equation}
 \lambda' (1+\overline \rho_{j,L}') \quad\text{for}\quad x \in (x_{j+1,L},x_{j,L}] \quad\text{and}\quad
 -\lambda' (1+\overline \rho_{i,R}') \quad\text{for}\quad x \in (x_{i,R},x_{i+1,R}] \end{equation}
for all $0 \leq j \leq k $ and $0 \leq i \leq \ell$. The curve $\eta'$ is then called the \emph{counterflow line} of $h$ starting from $0$. The reason for the differences in signs is that it allows one to couple flow and counterflow lines with the same GFF in such a way that the latter naturally grows in the opposite direction of the former.

In~\cite{ms2017ig4} this theory is extended to flow lines emanating from interior points. Let $h$ be a GFF on $\C$ defined modulo a global additive multiple of $2 \pi \chi$. There exists a coupling between $h$ and a a whole-plane $\SLE_\kappa(2-\kappa)$ curve $\eta$ started at $0$ such that the following is true. For any $\eta$-stopping time $\tau$, the conditional law of $h$ given $\eta|_{[0,\tau]}$ is given by that of the sum of a GFF $\wt{h}$ on $\C \setminus \eta([0,\tau])$ with zero boundary conditions and a random harmonic function $\Fh$ (which is defined modulo a global additive multiple of $2 \pi \chi$) on $\C \setminus \eta([0,\tau])$ whose boundary data is given by flow line boundary conditions on $\eta([0,\tau])$,~\cite[Theorem~1.1]{ms2017ig4}. Again, in this coupling $\eta$ is a.s.\ determined by $h$ \cite[Theorem~1.2]{ms2017ig4}.

\subsubsection{Interaction rules} 
\label{subsubsec:ig_interaction_rules}
In this work we will need to know how flow lines starting from different points and with different angles interact with each other. This is explained in ~\cite[Theorem~1.5]{ms2016ig1} (for boundary emanating curves) and in~\cite[Theorem~1.7]{ms2017ig4} (for general flow lines). 

Let $\theta_c= \pi \kappa/(4-\kappa)$ be the \emph{critical angle}. Suppose we are in the case of boundary emanating flow lines $\eta_1$, $\eta_2$ started from $x_1 \geq x_2$ with angles $\theta_1,\theta_2 \in \R$, respectively. In~\cite[Theorem~1.5]{ms2016ig1} it is shown that if $\theta_1<\theta_2$, then $\eta_1$ a.s.\ stays to the right of $\eta_2$. Moreover, if the angle gap $\theta_2-\theta_1$ is in $(0, \theta_c)$ then $\eta_1$ and $\eta_2$ can bounce off each other; otherwise the paths a.s.\ do not intersect at positive times. If the angle gap is zero then $\eta_1$ and $\eta_2$ merge with each others upon intersecting and do not subsequently separate. Finally, if the angle gap is in the range $(-\pi, 0)$ then $\eta_1$ may intersect $\eta_2$ and, upon intersecting, crosses and then never crosses back. If, in addition, $\theta_1-\theta_2 < \theta_c$, then $\eta_1$ and $\eta_2$ can bounce off each other after crossing; otherwise the paths a.s.\ do not subsequently intersect.

When $\eta_1$ and $\eta_2$ are interior emanating flow lines then $\eta_1$ can hit the right side of $\eta_2$ provided the angle gap between the two curves upon intersecting is in $(-\pi, \theta_c)$ and their subsequent behavior is the same as described above for flow lines starting from boundary points~\cite[Theorem~1.7]{ms2017ig4}.

\subsection{Space-filling $\SLE_{\kappa'}$}
\label{subsec:space_filling_sle}
 We now recall the construction and continuity of space-filling $\SLE_{\kappa'}$ in the context of imaginary geometry~\cite{ms2017ig4}.

We start with chordal space-filling $\SLE_{\kappa'}(\rho_1';\rho_2')$. Suppose that $h$ is a GFF on $\h$ with boundary conditions given by $\lambda'(1+\rho_1')$ on $\R_-$ and $-\lambda'(1+\rho_2')$ on $\R_+$. Consider a countable dense set $(r_k)$ in $\h$. Then we can give an ordering to the $(r_k)$ by saying that $r_j$ comes before $r_k$ if it is true that the flow line of $h$ starting from $r_j$ with angle $-\pi/2$ merges with the flow line of $h$ starting from $r_k$ with angle $-\pi/2$ on its right side. The interaction between these flow lines and the boundary is important in this construction. The rule is defined by viewing $\R_-$ and $\R_+$ as flow lines and using the interaction rules for flow lines~\cite[Theorem~1.7]{ms2017ig4} to see the behavior of these flow lines upon hitting $\partial \h$. The orientations of the two boundary segments $\R_-$ and $\R_+$ depend on the values of $\rho_1',\rho_2'$: for $\rho_j' \in (-2,\kappa'/2-4]$, the boundary segment is oriented towards $0$ from $\infty$ and if $\rho_j' \in (\kappa'/2-4,\kappa'/2-2)$, the boundary segment is oriented from $0$ towards $\infty$. It is shown in~\cite[Section 4]{ms2017ig4} that there is a unique space-filling curve that can be coupled with $h$ which respects this ordering and it is continuous when parameterized by Lebesgue measure. In the coupling the path is determined by the field $h$. Parameterizing a chordal space-filling $\SLE_{\kappa'}(\rho_1';\rho_2')$ using half-capacity then yields an ordinary chordal $\SLE_{\kappa'}(\rho_1';\rho_2')$ in $\h$ from $0$ to $\infty$.

The flow line interaction rules~\cite[Theorem~1.7]{ms2017ig4} imply the following: if we start a flow line from the target point of a chordal space-filling $\SLE_{\kappa'}(\rho_1';\rho_2')$ with angle in $[-\pi/2,\pi/2]$, the space-filling path will visit the range of the flow line in reverse chronological order. If we draw a counterflow line with the same starting and ending points as the space-filling $\SLE_{\kappa'}(\rho_1';\rho_2')$, then the space-filling $\SLE_{\kappa'}(\rho_1';\rho_2')$ will visit the points of the counterflow line in the same order and, whenever the counterflow line cuts off a component from $\infty$, the space-filling $\SLE$ fills up this component before continuing along the trajectory of the counterflow line.

 We will also consider whole-plane space-filling SLE$_{\kappa'}$ from $\infty$ to $\infty$, which was defined in~\cite[Section~1.4.1]{dms2021mating}.
 For $\kappa' \geq 8$, whole-plane space-filling SLE$_{\kappa'}$ from $\infty$ to $\infty$ is a curve from $\infty$ to $\infty$ which locally looks like an SLE$_{\kappa'}$. For $\kappa' \in (4,8)$, space-filling SLE$_{\kappa'}$ from $\infty$ to $\infty$ up until hitting~$0$ traces points in the same order as an SLE$_{\kappa'}(\kappa'-6)$ counterflow line of the same field from $\infty$ to~$0$, but fills in the regions which are between its left and right boundaries (viewed as a path in the universal cover of $\C \setminus \{0\}$) that it disconnects from its target point with a continuous space-filling loop. 
The whole-plane space-filling SLE$_{\kappa'}$ from $\infty$ to $\infty$ can be constructed again by starting with a fixed countable dense set of points $(r_k)$ in $\C$ and considering the flow lines $\eta^{\pm}_j$ of $h$ starting from the points $r_j$ with angle $\mp\pi/2$, where now $ h$ is a whole-plane GFF modulo a global additive multiple of $2\pi \chi$. These flow lines merge together forming the branches of a tree rooted at $\infty$. 
Again, we can define a total order on $(r_k)$ by deciding that $r_j$ comes before $r_k$ if and only if $r_k$ lies in a connected component of $\C\setminus (\eta_{j}^- \cup \eta_{j}^+)$ whose boundary is traced by the right side of $\eta_j^-$ and the left side of $\eta_j^+$. It can be shown (see~\cite[Footnote 4]{dms2021mating}) that there is a unique space-filling curve $\eta'$ which traces the points $(r_k)$ in order and is continuous when parameterized by Lebesgue measure. This path does not depend on the choice of $(r_k)$ and is defined to be whole-plane space-filling SLE$_{\kappa'}$ from $\infty$ to $\infty$.

\subsection{Distortion estimates for conformal maps}\label{sec:distortion-bounds} We briefly recall some distortion estimates for conformal maps.
		\begin{lemma}[Koebe-$1/4$ theorem]
			\label{lem:koebe}
			Suppose that $D \subseteq \mathbb C$ is a simply connected domain and $f \colon \D \to D$ is a conformal transformation. Then $D$ contains $B(f(0),|f'(0)|/4)$.
		\end{lemma}
		As a corollary of the Koebe-$1/4$ theorem one can deduce the following~\cite[Corollary~3.18]{lawler2005conformally}:
		\begin{lemma}
			\label{lem:deriv_bound}
			Suppose that $D,\wt{D} \subseteq \mathbb C$ are domains and $f \colon D \to \wt{D}$ is a conformal transformation. Fix $z \in D$ and let $\wt{z} = f(z)$. Then
			\[ \frac{\dist(\wt{z},\partial \wt{D})}{4\dist(z,\partial D)} \leq |f'(z)| \leq \frac{4 \dist(\wt{z},\partial \wt{D})}{\dist(z,\partial D)}.\]
		\end{lemma}
		Combining the Koebe-$1/4$ theorem and the growth theorem one can obtain the following~\cite[Corollary~3.23]{lawler2005conformally}:
		\begin{lemma}
			\label{lem:ball_bound}
			Suppose that $D,\wt{D} \subseteq \mathbb C$ are domains and $f \colon D \to \wt{D}$ is a conformal transformation. Fix $z \in D$ and let $\wt{z} = f(z)$. Then for all $r\in(0,1)$ and all $|w-z|\le r\dist(z, \partial D)$,
			\begin{align*}
				|f(w)-\wt z| \le \frac{4|w-z|}{1-r^2} \frac{\dist(\wt z, \partial \wt D)}{\dist(z, \partial D)} \le \frac{4r}{1-r^2}\dist(\wt z, \partial \wt D).
			\end{align*}
		\end{lemma}

\subsection*{Notation}

Let $\dist(\cdot,\cdot)$ be the Euclidean distance. Given a set $S\subseteq \C$ and $\ep>0$ we denote by $\neigh{\ep}{S}=\{x\in \C : \dist(x,S)<\ep\}$ and by $\ineigh{\ep}{S} =\{x\in S : \dist(x,\partial S)>\ep\}$. For a set $X$ we also let $\closure{X}$ denote its closure.

\section{Uniform estimates for flow line behavior}
\label{sec:uniform_flow_line_behavior}

In this section we derive some lemmas which allow us to control the behavior of flow lines uniformly in $\kappa$. We begin with the following uniform bound for the Radon-Nikodym derivative between the law of $\SLE_\kappa(\ul{\rho})$ processes where the force points agree in a neighborhood of their starting point.

\begin{lemma}
\label{lem:flow_line_abs_cont} 
Fix $\nu > 0$ and $\kappa \in [\nu,4-\nu]$. Suppose that $D_1$, $D_2$ are simply connected domains in $\C$ with $0\in \closure{D_1\cap D_2}$. For $i=1,2$, let $h_i$ be a zero boundary GFF on $D_i$ and let $F_i$ be a harmonic function on $D_i$ taking values in $[-1/\nu,1/\nu]$. Let $U\subseteq D_1\cap D_2$ be bounded, open, and such that $0 \in \closure{U}$. For $i=1,2$, let $\eta_i$ be the flow line of $h_i$ from $0$ to $\infty$ stopped upon either exiting $\closure{U}$ or hitting the continuation threshold.

\begin{enumerate}[(i)]
\item\label{lem:flow_line_abs_cont:interior}(Interior) Suppose there is $\zeta>0$ such that $\dist(U,\partial D_i)>\zeta$ for $i=1,2$. Then the law of $\eta_i$ is mutually absolutely continuous with respect to the law of $\eta_{3-i}$ for $i=1,2$.

\item\label{lem:flow_line_abs_cont:boundary}(Boundary) Suppose there is $\zeta>0$ such that for $ U'=\neigh{\zeta}{\closure U}$ it holds $\closure {D_1}\cap U'=\closure {D_2}\cap U'$, and that $(F_1- F_2)(z) \to 0$ as $z \to \partial D_i \cap U'$. Then the law of $\eta_i$ is mutually absolutely continuous with respect to the law of $\eta_{3-i}$ for $i=1,2$.

\end{enumerate}
In both cases, if $\CZ_i$ is the Radon-Nikodym derivative of the law of $\eta_i$ with respect to the law of $\eta_{3-i}$ then for each $p > 0$ there exists a constant $c > 0$ which depends only on $\nu$, $U$, and $\zeta$ such that $\E[ \CZ_i^p] \leq c$.

\end{lemma}

\begin{proof}
We will first give the proof of part~\eqref{lem:flow_line_abs_cont:interior}. Let $\wt U \subseteq D_1 \cap D_2$ be such that $\dist(\wt U ,\partial D_i) > \zeta/2$ for $i=1,2$, $U \subseteq \wt U $, and $\dist(U,\partial \wt U ) > \zeta/2$. We can write $(h_i+F_i)|_{\wt U } = h_i^{\wt U } + (h_i^{\wt U ^c}+F_i)|_{\wt U }$ where $h_i^{\wt U }$ is a zero boundary GFF on $\wt U $ and $(h_i^{\wt U ^c}+F_i)|_{\wt{U}}$ is harmonic. 
We can assume that $h_1,h_2$ are coupled together so that $h_1^{\wt U } = h_2^{\wt U }$ and that $h_1^{\wt U ^c}, h_2^{\wt U ^c}$ are independent. 
 Let $\phi \in C_0^\infty(\wt U )$ be such that $\phi|_U \equiv 1$ and let $g = \phi\cdot \left( (h_2^{\wt U ^c} + F_2) - (h_1^{\wt U ^c} + F_1) \right)$ so that $(h_1+F_1 + g)|_U= (h_2+F_2)|_U$. 
 Then $g$ is a measurable function of $h_1,h_2$ that is independent of $h_1^{\wt U } = h_2^{\wt U }$.
 The Radon-Nikodym derivative of the law of $(h_1+F_1+g)|_U $ with respect to $(h_1+F_1)|_U $ is given by $\CZ =\exp((h_1+F_1,g)_\nabla- \tfrac{1}{2} \| g \|_\nabla^2)=\exp((h_1^{\wt U },g)_\nabla-\tfrac{1}{2} \| g \|_\nabla^2)$. 
For each $p >0$ we have that
\begin{align*}
\E[ \CZ^p ]
=\E\left[ \exp\left( \frac{p^2-p}{2} \| g \|_\nabla^2 \right)\right].
\end{align*}
Note that the right hand side can be bounded in a way that depends only on $\zeta$, $U$, and $\nu$. Thus this completes the proof of part~\eqref{lem:flow_line_abs_cont:interior} as $\CZ_1 = \E[ \CZ \giv \eta_1]$ so that for $p \geq 1$ Jensen's inequality implies that $\E[\CZ_1^p] \leq \E[ \CZ^p]$.

We now turn to give the proof of part~\eqref{lem:flow_line_abs_cont:boundary}. Let $\wt U = U' \cap D_1 = U' \cap D_2$. Then we have $(h_i+F_i)|_{\wt U } = h_i^{\wt U } + (h_i^{\wt U ^c} + F_i)|_{\wt U }$ where $h_i^{\wt U }$ is a zero boundary GFF on $\wt U $ and $(h_i^{\wt U ^c}+F_i)|_{\wt U}$ is harmonic in $\wt U$. We consider $h_1,h_2$ to be coupled so that $h_1^{\wt U } = h_2^{\wt U }$ and $h_1^{\wt U ^c}, h_2^{\wt U ^c}$ are independent. Let $\phi$ be a $C^\infty$ function with $\phi|_U \equiv 1$, $\phi|_{\neigh{\zeta/2}{U}} \equiv 0$, and let $g = \phi \left( (h_2^{\wt U ^c}+F_2) - (h_1^{\wt U ^c}+F_1) \right)$. Then we have that $(h_1 + F_1 + g)|_U = (h_2 + F_2)|_U$ and the rest of the proof thus follows using the same argument used for part~\eqref{lem:flow_line_abs_cont:interior}.
\end{proof}

\begin{lemma}
\label{lem:uniform-far0}
Fix $\nu > 0$, $\kappa \in [\nu,4-\nu]$, $\rho^{0,L}, \rho^{0,R} \in [-2+\nu,1/\nu]$, and $x^{1,R} \in [\nu,1/\nu]$. Let $\eta$ be an $\SLE_\kappa(\rho^{0,L};\rho^{0,R},\rho^{1,R})$ process in $\h$ from $0$ to $\infty$ with force points located at $0_\pm$ and $x^{1,R}$. Let $\gamma : [0,1]\to \closure{\h}$ be a simple path with $\gamma(0) = 0$. Fix $\ep>0$ and set
\[ \sigma_1 = \inf\{ t \geq 0 : \eta(t) \notin \neigh{\ep}{\gamma}\}\quad\text{and}\quad
 \sigma_2 = \inf\{ t \geq 0 : |\eta(t) - \gamma(1)| \leq \epsilon\}.\]
There exists $p \in (0,1)$ depending only on $\nu$, $\gamma$, and $\epsilon$ such that $\p[ \sigma_2 < \sigma_1] \geq p$.
\end{lemma}
\begin{proof}
Choose $\delta \in (0,1/2)$ small compared to $\ep$ and $|\gamma(1)|$. Let $\tau$ be the first time $\eta$ exits $B(0,\delta/10)$ and let $\varphi$ be the unique conformal map from the unbounded component of $\h \setminus \eta([0,\tau])$ to $\h$ with $\varphi(\eta(\tau)) = 0$ and which also fixes $-1$ and $\infty$. We note that there exists a constant $c > 0$ such that a Brownian motion $B$ starting from $i$ has chance at least $c \delta$ of exiting $\h \setminus \eta([0,\tau])$ in $\eta([0,\tau])$. Indeed, this follows because~$B$ has chance at least a constant times~$\delta$ of exiting $\h \setminus \closure{B(0,\delta)}$ in the arc $\{ \delta e^{i \theta} : \theta \in [\pi/4,3\pi/4]\} \subseteq \h \cap \partial B(0,\delta)$ and given this it has a positive chance (not depending on $\delta$) of exiting $\h \setminus \eta([0,\tau])$ in $\eta([0,\tau])$. Let $z^{0,L}$ (resp.\ $z^{0,R}$) be the image under $\varphi$ of the leftmost (resp.\ rightmost) intersection of $\eta([0,\tau])$ with $\R_-$ (resp.\ $\R_+$). It follows that by possibly decreasing the value of $c > 0$ we have that either $z^{0,L} \leq -c \delta$ or $z^{0,R} \geq c \delta$.

We can choose some deterministic connected compact set $S\subseteq \closure{\h}$ depending only on $\gamma$, $\epsilon$, and~$\delta$ such that $S$ contains a neighborhood of $0$, $S \cap \partial\h$ contains at most one of $z^{0,L}$, $z^{0,R}$, $\varphi(\neigh{\ep}{\gamma})\setminus\{z\in \h : \im(z) \leq 2 \delta \} \subseteq \ineigh{\delta}{S}$. By Lemma~\ref{lem:flow_line_abs_cont} the law of $\eta$ stopped upon exiting $S_{\delta}$ is mutually absolutely continuous w.r.t.\ the law of an $\SLE_\kappa(\rho)$ process with a single force point stopped upon exiting $\ineigh{\delta}{S}$ where the Radon-Nikodym derivative has a second moment with a uniform bound depending only on $\nu$, $\gamma$, $\epsilon$, and $\delta$. We assume that we are in the case that $z^{0,R} \geq c \delta$ so that the process we are comparing to is an $\SLE_\kappa(\rho^{0,L})$ process in $\h$ from $0$ to $\infty$ with a single force point located at $z^{0,L}$ (the other case is analogous). By the Cauchy-Schwarz inequality it is sufficient to control the probability that an $\SLE_\kappa(\rho^{0,L})$ gets close to $\varphi(\gamma(1))$ before exiting $\varphi(\neigh{\ep}{\gamma})\cap \ineigh{\delta}{S}$. The distance of the force point $z^{0,L}$ to $0$ is of order $\delta$, but its exact location is random. For a fixed choice of~$\kappa$ and~$\rho^{0,L}$,~\cite[Lemma 2.3]{mw2017intersections} gives that an $\SLE_\kappa(\rho)$ process has a positive probability of the event above happening. The result follows because by Lemma~\ref{lem:bes-cont} this probability is continuous w.r.t.\ $\kappa$, $\rho^{0,L}$, and the location~$z^{0,L}$ of the force point.
\end{proof}

\begin{lemma}
\label{lem:uniform-far1}
Fix $\nu > 0$, $\kappa \in [\nu,4-\nu]$, $\rho^{0,L}, \rho^{0,R} \in [-2+\nu,1/\nu]$ and with $\rho^{0,R} + \rho^{1,R} \in [\kappa/2-4+\nu,\kappa/2-2-\nu]$. Let $\eta$ be an $\SLE_\kappa(\rho^{0,L};\rho^{0,R},\rho^{1,R})$ process in $\h$ from $0$ to $\infty$ with force points located at $0_\pm$ and $x^{1,R}$. Let $\gamma:[0,1]\to\closure{\h}$ be a simple path such that $\gamma((0,1))\subseteq \h$, $\gamma(0)=0$, and $\gamma(1) \in \partial \h$ with $\gamma(1) \geq x^{1,R} + \nu$. Fix $\ep>0$ and let
\[ \sigma_1 = \inf\{ t \geq 0 : \eta(t) \notin \neigh{\ep}{\gamma} \}\quad\text{and}\quad
 \sigma_2 = \inf\{ t \geq 0 : \eta(t)\in \partial\h\}.\]
There exists $p \in (0,1)$ depending only on $\nu$, $\gamma$, and $\epsilon$ such that $\p[ \sigma_2 < \sigma_1] \geq p$.
\end{lemma}

\begin{proof}
See Figure~\ref{fig:gamma-boundary} for an illustration of the setup. Choose $\delta \in (0,1/2)$ small compared to $\ep > 0$. Let $\tau$ be the first time $\eta$ exits $B(0,\delta/10)$ and let $\varphi$ be the unique conformal map from the unbounded component of $\h\setminus\eta([0,\tau])$ to $\h$ with $\varphi(\eta(\tau)) = 0$ and which fixes $-1$ and $\infty$. Let $z^{0,L}$ (resp.\ $z^{0,R}$) be the image under $\varphi$ of the leftmost (resp.\ rightmost) intersection of $\eta([0,\tau])$ with $\R_-$ (resp.\ $\R_+$). Arguing as in the proof of Lemma~\ref{lem:uniform-far0}, there exists a constant $c > 0$ so that we have that either $z^{0,L} \leq -c \delta$ or $z^{0,R} \geq c \delta$. We will assume that this is the case for $z^{0,R}$ (the argument in the case of $z^{0,L}$ is analogous). We can choose some deterministic connected compact set $S \subseteq \C$ depending only on $\gamma$, $\epsilon$, and $\delta$ such that~$S$ contains a neighborhood of the origin, $B(z^{0,R},c \delta/4)$ does not intersect~$S$, the distance between $S$ and $z^{1,R} =\varphi(x^{1,R})$ is bounded from below, and $\varphi(\neigh{\ep}{\gamma})$ minus a neighborhood of $z^{0,R}$ and $z^{1,R}$ is contained in $\ineigh{\delta}{S}$. By Lemma~\ref{lem:flow_line_abs_cont} the law of $\eta$ stopped upon exiting $S_{\delta}$ is mutually absolutely continuous w.r.t.\ the law of an $\SLE_\kappa(\rho^{0,L}; \rho^{0,R} + \rho^{1,R})$ process in $\h$ from $0$ to $\infty$ with force points at $z^{0,L}$, $z_*^{1,R}$ and stopped upon exiting $\ineigh{\delta}{S}$, where $z_*^{1,R}$ can be fixed deterministically (provided we choose $\delta > 0$ sufficiently small). Moreover the Radon-Nikodym derivative has a second moment with a uniform bound. Thus by the Cauchy-Schwarz inequality it is sufficient to control the probability that an $\SLE_\kappa(\rho^{0,L}; \rho^{0,R} + \rho^{1,R})$ hits $\partial\h$ before exiting $\varphi(\neigh{\ep}{\gamma})\cap \ineigh{\delta}{S}$.
The distance of $z^{0,L}$ to $0$ is of order $\delta$, but its exact location is random. To obtain a lower bound which is uniform in the location of the force point we note that by the Markov property for $\SLE_\kappa(\rho)$ curves and the continuity of the driving process (and choosing $\delta > 0$ small enough compared to $\ep$) it suffices to show a uniform lower bound on the probability of the event above when the force point in $\R_-$ is located at $0_-$. 
 By Lemma~\ref{lem:uniform-far0} the probability that the curve stays within distance $\delta/100$ of the vertical segment $[0, i\delta/5]$ up until it first hits $\partial B(0,\delta/5)$ has a positive lower bound which depends only on $\nu$. On the aforementioned event, we conformally map back again at the hitting time of $\partial B(0,\delta/5)$ with the tip sent to $0$ and $\infty$ fixed and we end up with an $\SLE_\kappa(\rho^{0,L} ; \rho^{0,R} + \rho^{1,R})$ process in $\h$ from $0$ to $\infty$ where $0$ has distance at least a constant times $\delta$ from both the force points. Again by absolute continuity, this allows us to reduce to the case of an $\SLE_\kappa(\rho)$ process in $\h$ from~$0$ to~$\infty$ with $\rho \in [\kappa/2-4+\nu,\kappa/2-2-\nu]$ and with force point $x^{R} \in \R_+$. 
 
 By decreasing the value of $\epsilon > 0$ if necessary we may assume that $\ep > 0$ is small enough that $\gamma(1) \geq x^R+2\ep$. Let $W^\eta$, $W^\gamma$, respectively, be the Loewner driving functions for $\eta$, $\gamma$ where we view both curves as being parameterized by capacity from $\infty$. Also, for a simple curve $\vartheta$ in $\h$ starting from $0$ and $y\in\partial\h\setminus\{0\}$ we will denote by $(g_t^{y,\vartheta})$ the chordal Loewner flow for $\vartheta$ as viewed from $y$.

 \begin{figure}[ht!]
			\begin{center}
				\includegraphics[scale=0.8]{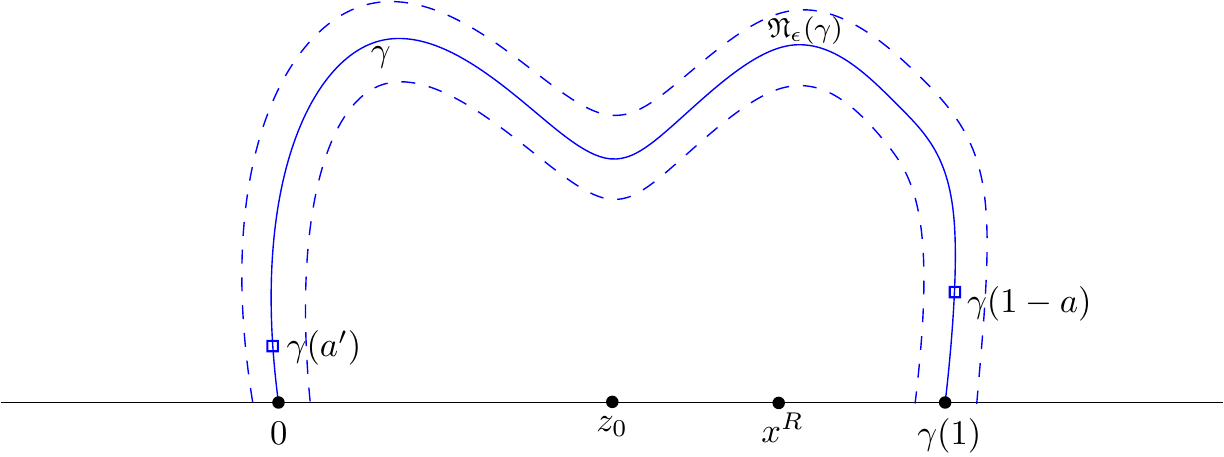}
				\caption {\label{fig:gamma-boundary}An illustration of some of the notation appearing in the proof of Lemma~\ref{lem:uniform-far1}.
				 The point $x^{R} \in \R_+$, the simple path $\gamma$ (with $\gamma((0,1))\subseteq \h$, $\gamma(0)=0$, $\gamma(1) \in \partial \h$ and $\gamma(1) \geq x^{R} + \nu$), and $\ep>0$, are given in the statement of the lemma. The dotted lines in the figure correspond to part of the boundary of $\neigh{\ep}{\gamma}$. We fix a point $z_0$ inside $ (0,x^{R}) \setminus \neigh{\ep}{\gamma}$ (possibly decreasing the value of $\ep$) and choose a time $a'>0$ so that $\neigh{\ep}{\gamma([0,a'])}$ does not disconnect $z_0$ from $\infty$. Given
				$\xi>0$, the time $a=a(\xi)\in(0,1-a')$ is chosen such that $|g_{1-a}^{\infty,\gamma}(\gamma(1))-W^\gamma_{1-a}|< \xi$, where $(g_t^{\infty,\gamma})$ the chordal Loewner flow for $\gamma$ as viewed from $\infty$ and $W^\gamma$ is the Loewner driving function for $\gamma$.}
				\end{center}
			\end{figure}

 Fix a point $z_0 \in (0,x^{R}) \setminus \neigh{\ep}{\gamma}$ (possibly considering some smaller $\ep > 0$ so that $(0,x^R) \setminus \neigh{\ep}{\gamma} \neq \emptyset$) and let $a' > 0$ be small enough that $\neigh{\ep}{\gamma([0,a'])}$ does not disconnect $z_0$ from $\infty$. Let $\xi>0$; we will choose its precise value later in the proof. Also, let $a=a(\xi)\in(0,1-a')$ be such that
 \begin{equation}
 \label{eqn:a_choice}
|g_{1-a}^{\infty,\gamma}(\gamma(1))-W^\gamma_{1-a}|< \xi. 	
 \end{equation}
Let $b=b(\xi)=\ep\wedge\dist(\gamma([a',1-a]),\partial\h)>0$. For each $\zeta > 0$, we can find $p > 0$ such that with
\[ E_a = \left\{ \sup_{t\in[0,1-a]} |W_t^\eta-W_t^\gamma| < \zeta \right\} \quad\text{we have}\quad \p[ E_a] \geq p\]
as $E_a$ is a positive probability event for the associated Brownian motion. Moreover, by Lemma~\ref{lem:bes-cont}, we can choose $p > 0$ so that it only depends on $\zeta, \nu$ (where we are using the first part of the statement of Lemma~\ref{lem:bes-cont}, as we have $\db(\kappa,\rho)\geq 4(1-1/\nu)$ and we are considering the event that the Bessel process associated with the driving pair of $\eta$ does not hit the origin in time $1-a$). Throughout the rest of the proof we will assume that we are working on $E_a$.

 We can assume $\zeta \in (0,\xi)$ and, by the definition of Carath\'eodory convergence and by~\cite[Proposition~4.43]{lawler2005conformally}, we can also choose $\zeta > 0$ small enough so that on the event $E_a$ we also have both
 \begin{enumerate}[(i)]
 \item\label{it:ea_extra1} $\eta([0,1-a])\subseteq \neigh{b/2}{\gamma([0,1-a])}$ (on this event $\eta([0,1-a])$ does not disconnect $z_0$ from $\infty$, since $\eta([0,a'])\subseteq\neigh{\ep}{\gamma([0,a'])}$ as $b \leq \ep$) and
 \item\label{it:ea_extra2} $|g^{\infty,\eta}_{1-a}(z_0)-g^{\infty,\gamma}_{1-a}(z_0)|<\xi$ (as we have the uniform convergence of $g_t^{\infty,\eta}(z)$ towards $g_t^{\infty,\gamma}(z)$ as $\zeta \to 0$ for $t\in[0,1-a]$ and $z$ bounded away from the hull of $\gamma([0,1-a])$).
 \end{enumerate}
In addition to $E_a$, we assume that we are working on both~\eqref{it:ea_extra1}, \eqref{it:ea_extra2} and study the evolution of $\eta$ after time $1-a$. Recalling~\eqref{eqn:a_choice} and using that $\gamma(1)$ is to the right of $z_0$ we also have that
 \begin{equation}
 \label{eqn:g_z_0_diff}
 |g^{\infty,\gamma}_{1-a}(z_0)-W^\gamma_{1-a}| \leq |g_{1-a}^{\infty,\gamma}(\gamma(1))-W^\gamma_{1-a}| \leq \xi.
 \end{equation}
Applying~\eqref{eqn:g_z_0_diff}, using that we are working on $E_a$, and recalling that $\zeta \in (0,\xi)$, we have by the triangle inequality that $|g^{\infty,\eta}_{1-a}(z_0)-W^\eta_{1-a}| \leq 3\xi$. By the Markov property, we have that $g_t^{\infty,\eta}(z_0) - g_t^{\infty,\eta}(\eta(t))$ for $t \geq 1-a$ evolves as $\sqrt \kappa$ times a Bessel process $X$ of dimension $\db(\kappa,\rho)<2$ (recall~\eqref{eqn:bessel_rho_formula}) started at $x_0 \in [0,3\xi/\sqrt\nu]$. We note that $\eta|_{[1-a,\infty)}$ needs some macroscopic capacity time $T(\ep)$, $\ep > 0$ depending only on $\ep$, to escape $\neigh{\ep}{\gamma}\setminus \neigh{\ep/2}{\gamma([0,1-a])}$ on the side of the boundary exposed to~$\infty$ (recall that $\eta([0,1-a]) \subseteq \neigh{b/2}{\gamma([0,1-a])}$ and $b \leq \ep$). Thus, if $\eta|_{[1-a,\infty)}$ were to exit $\neigh{\ep}{\gamma}$ on the side of the boundary of $\neigh{\ep}{\gamma}$ exposed to~$\infty$ before disconnecting $z_0$ from $\infty$, the process $X|_{[0, T(\ep)]}$ would fail to hit the origin. The probability $q_1(\xi)$ of this event can be made arbitrarily small by choosing $\xi$ small enough depending only on $\nu$ (since $q_1$ is continuous in $ \db(\kappa,\rho)$ which is continuous in $\kappa, \rho$ by \eqref{eqn:bessel_rho_formula}), see for example~\cite{Lawler2018NotesOT}.

 Suppose now that we parameterize $\eta$ and $\gamma$ by capacity as seen from $z_0$. By \cite[Theorem~3]{sw2005coordinate}, the former is distributed as an $\SLE_\kappa(\kappa-6-\rho,\rho)$ from $0$ to $z_0$ with force points at $\infty$ and $x^R$ up until the time $z_0$ is disconnected from $\infty$. We fix a compact domain $S'$ containing $\neigh{2\ep}{\gamma}$. The law of $\eta$ is mutually absolutely continuous w.r.t.\ the law of an $\SLE_\kappa(\kappa-6-\rho)$ process $\wt \eta$ in $\h$ from $0$ to $z_0$ when both curves are stopped at their exit times of $S'$ and by Lemma~\ref{lem:flow_line_abs_cont} the Radon-Nikodym derivatives have bounded second moment depending only on $\nu$ and $\ep$. Let $q_2(\xi)$ (resp.\ $\wt q_2(\xi)$) be the probability that $\eta|_{[1-a,\infty)}$ (resp.\ $\wt{\eta}|_{[1-a,\infty)}$) exits the portion of the boundary of $\neigh{\ep}{\gamma}$ exposed to~$z_0$ before disconnecting it from $\infty$. To show that $q_2(\xi)$ can be made arbitrarily small (depending only on $\nu$) as $\xi \to 0$ it is thus sufficient to show the analogous statement for $\wt q_2(\xi)$. This follows as above (possibly making $\zeta$ smaller to have $|g^{z_0,\wt \eta}_{1-a}(\infty)-g^{z_0,\gamma}_{1-a}(\infty)|<\xi$). Thus we can fix $\xi > 0$ small enough so that $q_1+q_2<1/2$, fix $\zeta$ accordingly, and obtain $p(\zeta,\nu)/2$ as the lower bound from the statement.
\end{proof}

The following is the $\kappa$-uniform version of~\cite[Lemma~3.10]{ms2017ig4}.

\begin{lemma}
\label{lem:boundary-merge}
Fix $\nu > 0$ and $\kappa \in [\nu,4-\nu]$. Suppose that $h$ is a GFF on $\h$ with constant boundary data $\lambda$. Let $\eta$ be the flow line of $h$ starting at $i$ and let $\tau = \inf\{t \in \R : \eta(t) \in \partial \h\}$. Let $E_1$ be the event that $\eta$ hits $\partial \h$ with a height difference of $0$ and let $E_2 = \{ \eta((-\infty,\tau]) \subseteq B(0,2)\}$. There exists $p \in (0,1)$ depending only on $\nu$ such that
\[ \p[E_1 \cap E_2] \geq p.\]
The same holds when $\lambda$ is replaced with $-\lambda$. 
\end{lemma}
\begin{proof}

{\it Step 1: Control on the winding number.} Let $\tau_{1/2}$ be the first time that $\eta$ hits $\partial B(i,1/2)$. The first step in the proof is to control the number of times $\eta|_{(\tau_{1/2},\tau]}$ needs to wind around $i$ to obtain height difference $0$ upon hitting $\partial \h$.

 Let $h_\ep$ be a GFF in $\h\setminus B(i, \ep)$ with constant boundary value $\lambda$ on $\R$ and boundary conditions on $\partial B(i, \ep)$ given by $-\lambda'+\chi\cdot$(winding of $\partial B(i, \ep)$) (resp.\ $\lambda'+\chi\cdot$(winding of $\partial B(i, \ep)$)) on the clockwise (resp.\ counterclockwise) arc between $i(1-\ep)$ and $i(1+\ep)$, see~\cite[Proof of Theorem~1.1]{ms2017ig4}. Let $\eta_\ep$ be the flow line of $h_\ep$ starting from $i(1+\ep)$. Let $T_\ep$ be the capacity of $B(i,\ep)$ (viewed as a subset of $\C$ and as seen from $\infty$) and suppose we parameterize $\eta_\ep|_{[T_\ep,\infty)}$ so that $\eta_\ep([T_\ep,t])\cup B(i,\ep)$ has capacity~$t$ as seen from $\infty$. Let $\tau^\ep_{1/2}$ be the exit time of $\eta_\ep$ from $B(i,1/2)$ and
 for every $T_\ep<t\leq\tau^\ep_{1/2}$ let $F^\ep_t$ be the unique conformal map from the connected component of $\h\setminus (B(i, \ep)\cup\eta_\ep([T_\ep,t]))$ with~$\R$ on its boundary to the annulus of outer radius $1$ and centered at $0$ that maps $0$ to $-i$.
 Let us denote the range of $F^\ep_t$ by $A=A^\ep_t$ and by $r=r_t^\ep$ the inner radius of $A$. Since the domain $A$ of $(F_t^\ep)^{-1}$ is not simply connected we need to be careful with our definition of $\arg ((F^\ep_t)^{-1})'$, which we make sense of in the following way. Let $\vartheta$ be a smooth simple curve in the component of $\h\setminus(\eta_\ep([T_\ep,t])\cup B(i,\ep))$ with $\R$ on its boundary starting at $\eta_\ep(t)$ which goes in the counterclockwise direction (winding around $i$ a minimal number of times) until hitting $\R$. Let $\eta^*_{\ep,t}$ be the set given by the the union of $\partial B(i,\ep)$ and the concatenation of $\eta_\ep|_{(-\infty,t]}$ with $\vartheta$ and let $C_\ep$ be the connected component of $\h\setminus \eta^*_{\ep,t}$ with $\R$ on its boundary. Then $\arg ((F^\ep_t)^{-1})'{|_{C_t^\ep}}$ is a well-defined function on $F^\ep_t(C_\ep)$ (since this is a simply connected domain) and by $\arg ((F^\ep_t)^{-1})'$ we mean the harmonic function on $A$ having the boundary conditions given by those of $\arg ((F^\ep_t)^{-1})'|_{C_t^\ep}$ on $\partial A$.
 
Let $\wt W^\ep_t,\wt O^\ep_t$ be the images under $F^\ep_t$ of the tip $\eta_\ep(t)$ and of the force point, respectively. In other words, $\wt{O}_t^\ep$ is the image under $F_t^\ep$ of the most recent intersection of $\eta^\ep$ before time $t$ with itself or, if there is no such intersection, then $i(1-\ep)$. Then conditionally on $\eta_\ep|_{[T_\ep,t]}$, the field $h_\epsilon \circ (F_t^\epsilon)^{-1} - \chi \arg ((F_t^\epsilon)^{-1})'$ is a GFF on $A$ with boundary values $\lambda + \chi \cdot \text{winding of }\partial\D$ (with the normalization that the boundary conditions are $\lambda$ at $-i$) on the outer circle, and with boundary conditions on the inner circle given by
\[-\lambda' + \chi \cdot\text{winding of }\partial B(0,r) + 2\pi\chi \cdot N^\ep_{t}\quad\text{and}\quad \lambda' + \chi\cdot\text{winding of }\partial B(0,r) + 2\pi\chi \cdot N^\ep_{t},\]
 on the clockwise and counterclockwise arc from $\wt O^\ep_{t}$ to $\wt W^\ep_{t}$, respectively, where $N^\ep_{t}$ denotes the net winding of $\eta_\ep|_{[T_\ep,t]}$ around $i$ counting negatively (resp.\ positively) when $\eta_\ep$ winds clockwise (resp.\ counterclockwise). 
In fact, the number of times that a flow line from $\wt W^\ep_{t}$ in $A $ would need to wind around the inner circle to obtain height values in the correct range to hit and merge into $\partial \D$ is the same as the corresponding number for $\eta_\ep|_{[t,\infty)}$, which is given by $N^\ep_{t}$.

Let us now consider a GFF $h$ on $\h$ with constant boundary conditions $\lambda$. Let $\eta$ be the flow line started from $i$, let $\tau_{1/2}$ be its exit time from $B(i,1/2)$, and for $t\in(-\infty,\tau_{1/2}]$ let $F_t$ be the conformal map from $\h\setminus\eta((-\infty,t])$ to the annulus of outer radius $1$ centered at $0$, such that the image of $0$ is $-i$. We also let $(\wt W,\wt O)$ be the images under $F$ of the tip and the force point, respectively.

Let $\wh{h}_\ep$ be a GFF in $\C\setminus B(i, \ep)$ with the same boundary conditions as $h_\ep$ on $\partial B(i, \ep)$ and $\wh{h}$ be a whole-plane GFF defined modulo a global multiple of $2\pi\chi$. Let $\wh{\eta}_\ep|_{[T_\ep,t]}$ (resp.\ $\wh{\eta}|_{(-\infty,t]}$) be the flow line of $\wh{h}_\ep$ (resp.\ $\wh{h}$) started from $i(1+\ep)$ (resp.\ $i$). By \cite[Propositions 2.1, 2.10]{ms2017ig4} the law of $\wh{h}_\ep$ converges as $\ep \to 0$ to the law of $\wh{h}$ when both are seen as distributions defined up to a global multiple of $2\pi\chi$. Using Lemma~\ref{lem:flow_line_abs_cont}, the law of the pair $(h_\ep,h)$ can be locally compared to that of $(\wh{h}_\ep,\wh{h})$ when this pair is viewed modulo a common additive constant. 
The second moments of the Radon-Nikodym derivatives between two pairs are controlled uniformly in $\ep$. Since the flow lines are determined by the field, we obtain that $(\wt W^\ep,\wt O^\ep)$ converges weakly to $(\wt W,\wt O)$ w.r.t.\ the topology of local uniform convergence and the boundary conditions on the inner circle converge.

 Thus letting $t=\tau_{1/2}$ and
 $\wt{h} = h \circ F^{-1} - \chi \arg (F^{-1})’ $ be the corresponding GFF on the annulus, we can make sense of $\arg((F^{-1})')$ as above and see that this is compatible with $\wt{h}$ having boundary conditions given by
 $\lambda + \chi \cdot \text{winding of }\partial\D$ on the outer circle (normalized so that the boundary conditions are $\lambda$ at $-i$) and on the inner circle the boundary values are
 \[-\lambda' + \chi \cdot\text{winding of }\partial B(0,r) + 2\pi\chi \cdot N_{1/2}\quad\text{and}\quad\lambda' + \chi\cdot\text{winding of }\partial B(0,r) + 2\pi\chi \cdot N_{1/2} .\]
 on the clockwise and counterclockwise arc from $\wt O_{\tau_{1/2}}$ to $\wt W_{\tau_{1/2}}$, respectively, where $N_{1/2}$ is the number of times that a flow line from $\wt W_{\tau_{1/2}}$ in $A $ needs to wind around the inner circle to obtain height values in the correct range to hit and merge into $\partial \D$.

 Fix $\phi\in C^\infty$ such that $\phi|_{B(i,3/4)}\equiv0$, $\phi\geq 0$, $\phi|_{B(i,7/8)\setminus B(i,13/16)}\equiv 1$ and $\int \phi(z) dz = 1$. Then the integral $(h,\phi)$ is a Gaussian random variable with mean $\lambda$ (hence bounded by a $\nu$-dependent constant) and has bounded variance. By applying a change of coordinates, we have that
\[(h, \phi) = (h \circ F^{-1}, \phi \circ F^{-1} |(F^{-1})'|^2) = (\wt{h} + \chi \arg (F^{-1})’, \phi \circ F^{-1} |({F^{-1}})'|^2).\]
We can write $\wt h= \mathfrak h_1+\mathfrak h_2+\wt h^0$ where $\wt h^0$ is a zero boundary GFF in $A$, $ \mathfrak h_1$ is harmonic with boundary conditions given by $2\pi\chi N_{1/2}$ on the inner circle and $0$ on the outer circle, $\mathfrak h_2$ is harmonic with boundary conditions given by 
$\lambda + \chi \cdot (\text{winding of }\partial\D)$
on the outer circle and by \[-\lambda' + \chi \cdot\text{winding of }\partial B(0,r), \quad \lambda' + \chi\cdot\text{winding of }\partial B(0,r) \]
 on the clockwise and counterclockwise arc from $\wt O_{\tau_{1/2}}$ to $\wt W_{\tau _{1/2}}$, respectively. 

We explicitly have that $\mathfrak h_1(z)=2\pi\chi N_{1/2}\log |z| /\log(r)$ where $r$ is the inner radius of $A$. Let~$r_0$ be the inner radius we would obtain by conformally mapping $\h\setminus B(i,1/2)$ to an annulus centered at $0$ with outer radius equal to $1$. Then $r\leq r_0$ and moreover there exists some $r_1 > 0$ such that $F(B(i,3/4)^c)\subseteq B(0,r_1)^c$, thus $|\mathfrak h_1(z)| \leq 2\pi\chi |N_{1/2}|\log(r_1)/\log(r_0)$ for all $z\in F(B(i,3/4)^c)$. Considering disconnection probabilities for a Brownian motion in $\h\setminus B(i, 1/2)$ we see that there exists $r_2>0$ such that $r\geq r_2$ a.s.\ There also exists some $r_3 <1$ such that $F(B(i,7/8)\setminus B(i,13/16))\subseteq B(0,r_3)$, hence $|\mathfrak h_1(z)|\geq 2\pi\chi |N_{1/2}|\log(r_3)/\log(r_2)$ for all $z\in F(B(i,7/8)\setminus B(i,13/16))$. Note that $\chi\in[\nu/4, (4-\nu)/(2\sqrt\nu)]$ for all $\kappa \in[\nu, 4 - \nu]$, thus we obtain that there exist $c_\nu, C_\nu>0$ depending only on $\nu$ so that 
 \[\left|(\mathfrak h_1,\phi \circ F^{-1} |(F^{-1})'|^2)\right|\in\left[c_\nu |N_{1/2}|, C_\nu |N_{1/2}|\right].\]
The term $(\wt h^0+\mathfrak h_2, \phi \circ F^{-1} |(F^{-1})'|^2)$ is a Gaussian random variable with mean bounded depending only on $\nu$ and bounded variance (using the bounds in Section~\ref{sec:distortion-bounds} to control $F'$ on the support of~$\phi$), and the term $ ( \chi \arg (F^{-1})', \phi \circ F^{-1} |(F^{-1})'|^2) $ is bounded. Thus $N_{1/2}$ has a Gaussian tail with bounds depending only on $\nu$ so we can fix $K$ deterministic depending only on $\nu$ such that with positive probability depending only on $\nu$ the number of times that $\eta$ needs to wind around $i$ after first exiting $B(i,1/2)$ is in $[-K,K]$.

{\it Step 2: Control on the behavior of $\eta$.} We now show that with positive probability depending only on $\nu$ the curve winds the correct amount of times around $i$ so as to hit~$\R$ with a height difference of~$0$ before leaving $\h \cap B(0,2)$. For every $\ep>0$, conditionally on $\eta((-\infty,\tau_{1/2}])$ let $\ell$ be a radial line segment from $\partial B(i,1/2)$ to $\partial B(i,3/4)$ starting within distance $\ep/10$ to $\eta(\tau_{1/2})$. We claim that there exists $p > 0$ depending only on $\nu$, $\epsilon$ so the probability that $\p[ \dist(\eta([\tau_{1/2},\tau_{3/4}]),\ell) \leq \epsilon] \geq p$.

Let us denote by $z_*$ the last intersection point of $\eta$ with itself before time $\tau_{1/2}$.
Conditionally on $\eta|_{(-\infty, \tau_{1/2}]}$, Lemma~\ref{lem:flow_line_abs_cont} implies that the law of $\eta|_{[\tau_{1/2},\tau_{3/4}]}$ is mutually absolutely continuous w.r.t. the law of a radial $\SLE_\kappa(2-\kappa)$ process $\wh\eta$ in the unbounded connected component of $\C\setminus\eta|_{(-\infty, \tau_{1/2}]}$ started from $\eta(\tau_{1/2})$ to $\infty$ with force point at $z_*$, and stopped upon exiting $B(i,3/4)$. Moreover the Radon-Nikodym derivatives have a second moment bounded by a constant depending only on $\nu$. Thus, by the Cauchy-Schwarz inequality, it is enough to prove the claim for the curve $\wh\eta|_{[\tau_{1/2},\tau_{3/4}]}$ where the stopping times $\tau_{1/2},\tau_{3/4}$ now correspond to the curve given by the concatenation of $\eta|_{(-\infty, \tau_{1/2}]}$ and $\wh\eta$.
Let $z_0$ be the point on $\partial B(i,3/4)$ having the same argument as $\eta(\tau_{1/2})$. 
We consider the conformal map $\varphi$ from the unbounded connected component of $\C\setminus\eta((-\infty, \tau_{1/2}])$ to $\h$, mapping $\eta(\tau_{1/2})$ to $0$ and $z_0$ to $i$. 
The conditional law of the image of $\wh\eta|_{[\tau_{1/2},\tau_{3/4}]}$ under $\varphi$ coincides with that of a chordal $\SLE_\kappa(2-\kappa,2 \kappa-8)$ in $\h$ from $0$ to $\infty$ with one force point at $\varphi(z_*)$ and additional force point at $\varphi(\infty)$, when both processes are stopped upon disconnecting $\varphi(\infty)$ from~$\infty$ \cite[Theorem~3]{sw2005coordinate}.
There exists a deterministic constant $C > 0$ such that $\varphi(B(i,3/4)\cap B(\wh\eta(\tau_{1/2}),3/8))$ is contained in $B(0,C)\cap\h$, and $B(0,C)\cap\h$ does not contain $\varphi(\infty)$. Thus, by Lemma~\ref{lem:flow_line_abs_cont} and the Cauchy-Schwarz inequality it is sufficient to consider the case of a chordal $\SLE_\kappa(2-\kappa)$ in~$\h$ from~$0$ to~$\infty$ with one force point at $\varphi(z_*)$ stopped upon hitting $\partial B(0,C)\cap\h$. 
By Lemma~\ref{lem:bes-cont}, the probability of the Loewner driving function of $\eta$ being sufficiently close to $0$ up until the hitting time of $\partial B(0,C)\cap\h$ by $\eta$ is continuous w.r.t.\ $\kappa$, $\rho$, and the location $\varphi(z_*)$ of the force point.
Thus we can fix $\delta>0$ and obtain a positive lower bound only depending on $\nu$ in the case $\varphi(z_*)\in[-\delta,\delta]$. When $\varphi(z_*)\notin[-\delta,\delta]$ we can use the usual argument from Lemma~\ref{lem:flow_line_abs_cont} to reduce to the case of an ordinary $\SLE_\kappa$ process, and we also obtain a positive lower bound depending only on $\nu$.

Note that for every $\ep$ we can reduce the choices of $\ell$ to a finite deterministic set $L_\ep$. We can choose $\ep$ as a deterministic function of $K$ to be small enough so that for every $j=-K,\ldots,K$ and every $\ell$ in $L_\ep$ we can draw a path $\gamma^j$ that starts from the endpoint of $\ell$, winds $j$ times around $B(i,3/4)$, and having neighborhood $\neigh{\ep}{\gamma^j}$ such that it is contained in $B(0,2)$, its boundary does not intersect itself and each side of its boundary also winds around $B(i,3/4)$ exactly $j$ times.
For each $\ell$ and $j$ we can fix a deterministic $\gamma^j$.
We argue that for any $\ell$ and $j$, on the event that $\eta|_{[\tau_{1/2},\tau_{3/4}]}$ travels inside $\neigh{\ep}{\ell}$, the probability of $\eta$ traveling inside $\neigh{\ep}{\gamma^j}$ has a positive lower bound depending only on $\ep$ and $\nu$. Since we have finitely many choices for the paths $\gamma^j(\ell)$ as $j=-K,\dots,K$, and for the segments $\ell\in L_\ep$, we obtain a positive lower bound depending only on $\ep$ and $\nu$.
From now on let $\gamma=\gamma^j$. Let us fix $x_0 = \gamma(1)$, and let $\gamma^L$ (resp.\ $\gamma^R$) be a simple path in $\neigh{\ep}{\gamma}\setminus (\eta((-\infty,\tau_{3/4}]) \cup \gamma)$ which connects a point on the left (resp.\ right) side of $\eta((\tau_{1/2},\tau_{3/4})) \cap \neigh{\ep}{\gamma}$ to a point on the same side of $\partial \h$ hit by $\gamma$ at time $1$, say $x^L$ (resp.\ $x^R$), and does not intersect $\neigh{3\ep/4}{\gamma}$ (note that we can assume their union to be deterministic). Assume that $\gamma^L \cap \gamma^R = \emptyset$. Let $S$ be the region of $\h\setminus \eta ((-\infty,\tau_{3/4}])$ which is surrounded by $\gamma^L$, $\gamma^R$, and $\partial\h$. Let $\wt{h}$ be a GFF on $S$ whose boundary data agrees with that of $h$ on $\eta ((-\infty,\tau_{3/4}])$ and on $\partial \h$ and is otherwise given by flow line boundary conditions. We choose the angles of the boundary data on $\gamma^L,\gamma^R$ so that the flow line $\wt{\eta}$ of $\wt{h}$ starting from $\eta (\tau_{3/4})$ is an $\SLE_\kappa(\rho^L;\rho^R)$ process targeted at $x_0$ and the force points are located at~$x^L$ and~$x^R$. Moreover, $\rho^L, \rho^R \in (\tfrac{\kappa}{2}-4,\tfrac{\kappa}{2}-2)$ since we assumed that if we continued the boundary data for $h$ given $\eta|_{(-\infty,\tau_{3/4}]}$ along $\gamma$ as if it were a flow line then it is in the admissible range for hitting (and the same is true for both $\gamma^L$ and $\gamma^R$ by construction). 
By Lemma~\ref{lem:flow_line_abs_cont}, it suffices to show that with uniformly positive probability $\wt \eta$ reaches $ \partial \h$ without getting within distance ${\ep/4}$ from $\gamma^R$ and $\gamma^L$. This is given by Lemma~\ref{lem:uniform-far1}.
\end{proof}

Recall from \cite[Theorem~1.5]{ms2016ig1} that when $h$ is a GFF in $\h$ with piecewise constant boundary data and $\theta\in[0,\tfrac{\pi\kappa}{4-\kappa})$ (recall that $\tfrac{\pi\kappa}{4-\kappa}$ is the critical angle as described in the beginning of Section~\ref{subsubsec:ig_interaction_rules}) then the flow line $\eta_\theta$ of $h$ starting from $0$ with angle $\theta$ is to the left of the flow line $\eta$ of $h$ starting from $0$ with angle $0$ and $\eta_\theta$ intersects $\eta$ on the left side of $\eta$.

 \begin{lemma}
\label{lem:good-pockets}
Fix $\nu > 0$ and let $\kappa \in [\nu,4-\nu]$. Suppose that $h$ is a GFF on $\h$ with boundary data $-\lambda$ on $\R_-$ and $\lambda$ on $\R_+$. Let $\theta=\tfrac 1 2 \tfrac{\pi\kappa}{4-\kappa}$ be half of the critical angle. Let $\eta$ be the flow line of~$h$ from~$0$ and let~$\eta_\theta$ be the flow line starting from~$0$ with angle~$\theta$. Let $\tau$ (resp.\ $\tau_\theta$) be the first time that~$\eta$ (resp.\ $\eta_\theta$) leaves $B(0,2)$. Let $E$ be the event that $A = \eta([0,\tau]) \cup \eta_\theta([0,\tau_\theta])$ separates $i$ from $\infty$ and that the harmonic measure of the left side of $\eta$ as seen from $i$ in $\h \setminus A$ is at least $\tfrac{1}{4}$. There exists $p \in (0,1)$ depending only on $\nu$ such that $\p[E] \geq p$.
\end{lemma}
 \begin{proof} Let $\gamma_1$ be the straight line $\{re^{i3\pi/8}:r\geq0\}$. 
 By Lemma~\ref{lem:uniform-far0} there exists $p > 0$ depending only on $\nu$ so that $\eta([0,\tau]) \subseteq \neigh{1/100}{\gamma_1}$. Let $\varphi$ be the unique conformal map from $\h \setminus \eta([0,\tau])$ to $\h$ that maps $0_-$ to $0$ (note that since $\eta$ is an $\SLE_\kappa$ it does not intersect the boundary other than at $0$), $\eta(\tau)$ to $2$, and fixes $\infty$.
 There exist a deterministic constant $c>0$ and a deterministic simple curve $\gamma_2:[0,1]\to\closure{\h}$ with $\gamma_2((0,1))\subseteq\h$, $\gamma_2(0)=0$, $\gamma_2(1)=1$, $\neigh{2c}{\gamma_2}$ contained in $\varphi(B(0,2))$, $B(\varphi(i),2c)$ contained in the bounded connected component of $\h\setminus\gamma_2((0,1))$, and such that the harmonic measure of the right side of $\gamma_2$ seen from $\varphi(i)$ is less than $1/2$. Using Lemma~\ref{lem:flow_line_abs_cont} as in the proof of the previous lemmas when considering $\varphi(\eta_\theta)$ we can disregard the force point coming from the tip $\eta(\tau)$ (since its image is bounded away from the origin). To obtain a uniform lower bound on the probability that $\varphi(\eta_\theta)$ stays sufficiently close to $\gamma_2$ up until hitting $\varphi(\eta([0,\tau]))$ we can thus reduce to the case of an $\SLE_\kappa(\rho^{0,L};\rho^{0,R})$ with $\rho^{0,L}=-\frac \kappa 4$ and $\rho^{0,R}=\frac \kappa 4 -2$. As $\rho^{0,L}>-2+\nu/4$, and $-2+\nu/4<\rho^{0,R}<\kappa/2-2-\nu/4$ for every $\kappa \in [\nu,4-\nu]$, the statement follows from Lemma~\ref{lem:uniform-far1} (with $\rho^{1,R}=0$).
 \end{proof}


\section{Uniform estimates for filling a ball and pocket diameter}
\label{sec:ball_pocket}

The purpose of this section is to prove uniform estimates for the probability that a whole-plane space-filling $\SLE_{\kappa'}$ from $\infty$ to $\infty$ with $\kappa' \in [4+\nu,8)$ fills in a ball of radius $\delta \epsilon$ when traveling distance~$\epsilon >0$ for $\delta \in (0,1)$ (Lemma~\ref{lem:bubble-hitting-0} in Section~\ref{subsec:fill_ball}) and for the size of the pockets which are formed by the branches of the trees of flow/dual flow lines around the origin (Lemma~\ref{lem:large-pocket} in Section~\ref{subsec:fill_ball}) and starting from a fine grid of points (Lemma~\ref{lem:4.14-unif} in Section~\ref{subsec:pocket_diameter}). The rest of the paper can be read independently of these statements and so the proofs given in this section can be skipped on a first reading.

\subsection{Filling a ball}
\label{subsec:fill_ball}

Fix $\kappa' \in (4,8)$ and let $\eta'$ be a whole-plane space-filling SLE$_{\kappa'}$ from $\infty$ to $\infty$ parameterized so that $\eta'(0) = 0$. By \cite[Lemma~3.1]{ghm2020kpz}, for each fixed $\kappa' \in (4,8)$ we have that $\eta'$ is very unlikely to travel a given distance without filling a Euclidean ball with diameter proportional to the distance traveled. The purpose of the first part of this subsection is to show (Lemma~\ref{lem:bubble-hitting-0}) that for each $\nu > 0$ this estimate can be made to be uniform in $\kappa' \in [4+\nu,8)$. To this end, we will use the relationship between the whole-plane space-filling SLE$_{\kappa'}$ and whole-plane SLE$_{\kappa'}(\kappa'-6)$ counterflow lines from $\infty$ coupled together using a common whole-plane GFF with values modulo a global multiple of $2\pi \chi$. (Our proof might seem somewhat heavy handed in that we will make use of the reversibility of $\SLE_{\kappa'}$ for $\kappa' \in (4,8)$ \cite{ms2016ig3,ms2017ig4} which is itself a non-trivial result. It is possible to give a proof which does not use reversibility but we chose to take this route in order to simplify the exposition.)

\newcommand{\sfwp}{\eta'}
\newcommand{\ksixwp}{\eta'_{0,\C}}
\newcommand{\ksixhp}{\eta'_{0,\h}}
\newcommand{\ksixhpz}[1]{\eta'_{#1,\h}}
\newcommand{\slehp}{\wt\eta'}
\newcommand{\timeksixwp}[1]{S_{#1}}
\newcommand{\timeksixhp}[1]{S^{\h}_{#1}}

 Let $\ksixwp$ be a whole-plane SLE$_{\kappa'}(\kappa'-6)$ from $0$ to $\infty$. For each $r > 0$ we let
\begin{equation}
\label{eqn:s_r_def}
\timeksixwp{r} := \inf\left\{ s>-\infty : |\ksixwp(s)| = r \right\} .
\end{equation}
 
Fix $b>a>0$ and $r>0$. We will study the event that $\ksixwp([\timeksixwp{a},\timeksixwp{b}])$ \emph{disconnects a bubble} of radius at least $r$ from $\infty$ and $0$. By this we mean the following: there exists $\timeksixwp{a} \leq s < t \leq \timeksixwp{b}$ such that $\ksixwp([s,t])$ does not disconnect $0$ from $\infty$ and the set given by the closure of the union of $\ksixwp([s,t])$ with the bounded connected components of its complement contains a ball of radius $r$.

\begin{lemma}
\label{lem:ball-0}
There exists a constant $\delta \in (0,1)$ such that the following is true. Fix $\nu > 0$ and $\kappa' \in [4+\nu,8)$. There exists $p\in(0,1)$ depending only on $\nu$ such that for every $\ep\in(0,1)$ it a.s.\ holds with conditional probability at least $p$ given $ \ksixwp|_{(-\infty,\timeksixwp{1}]}$ that $ \ksixwp([\timeksixwp{1},\timeksixwp{1+\ep} ])$ disconnects a bubble of radius at least $\delta \ep$ from $\infty$ and $0$.
\end{lemma}

Given $ \ksixwp|_{(-\infty, \timeksixwp{1}]}$, the rest of the curve $ \ksixwp$ is distributed as a radial SLE$_{\kappa'}(\kappa'-6)$ from $ \ksixwp(\timeksixwp{1})$ to $\infty$ in the unbounded component of $\C \setminus \ksixwp((-\infty, \timeksixwp{1}])$ with force point $x^*$ at the last intersection point of $\ksixwp|_{(-\infty, \timeksixwp{1}]}$ with itself (before time $\timeksixwp{1}$). 
Let $z_0$ be the point on $\partial B(0,1+\ep)$ with the same argument as $\eta_0'(\timeksixwp{1})$ and let $\varphi$ be the unique conformal map from the unbounded component of $\C \setminus \ksixwp((-\infty,\timeksixwp{1}])$ to $\h$ with $\varphi(\eta_0'(\timeksixwp{1})) = 0$ and $\varphi(z_0) = i$. 
 By the target invariance of $\SLE_{\kappa'}(\kappa'-6)$, the image of $ \ksixwp$ is distributed as a chordal SLE$_{\kappa'}(\kappa'-6)$ from $0$ to $\infty$ with force point at $\varphi(x^*)$ up until the time that $\varphi(\infty)$ is disconnected from $\infty$. Let us denote this curve by $ \ksixhp$. For $r>0$ we denote by $\timeksixhp{r}$ the exit time of $ \ksixhp$ from $\h \cap B(0,r)$.

We first argue that there exists some $r_0 \in(0,1)$ so that $\varphi^{-1}(\h \cap B(0,r_0)) \subseteq B(0,1+\ep)$, independently of $\ep$. Let $r\in(0,1)$. On the one hand, the harmonic measure of $\h \cap B(0,r)$ inside $\h$ seen from $i$ is $O(r)$ as $r\to0$. On the other hand, the harmonic measure of $\varphi^{-1}(\h \cap B(0,r))$ in $\C \setminus \ksixwp((-\infty, \timeksixwp{1}])$ seen from $z_0$ is bounded below by the harmonic measure of $\varphi^{-1}(\h \cap B(0,r))$ seen from $z_0$ inside $B(z_0,2\ep)$ and $\varphi^{-1}(\h \cap B(0,r))\cap B(z_0,2\ep)$ has a unique connected component with $\ksixwp(\timeksixwp{1})$ on its boundary.
There exists a constant $c>0$ independent of $\ep$ such that the harmonic measure from $z_0$ inside $B(z_0,2\ep)$ of any connected set $K$ containing $ \ksixwp(\timeksixwp{1})$ that intersects the boundary of $B(z_0,2\ep)\cap B(0,1+\ep)$ 
is bounded below by $c$. Thus, for $r_0 > 0$ small enough independently of $\ve$, the set $\varphi^{-1}(\h \cap B(0,r_0))\cap B(z_0,2\ep)$ cannot intersect the boundary of $B(z_0,2\ep)\cap B(0,1+\ep)$. That is, $\varphi^{-1}(\h \cap B(0,r_0)) \subseteq B(0,1+\ep)$ for any $\ep > 0$.

 Thus it suffices to show that, for some $z\in B(0,{1+\ep})$ and $\delta>0$, the curve $ \ksixhp$ stopped upon exiting $B(0,r_0)$ disconnects 
 $\varphi(B(z,\delta\ep))$ both from $\varphi(x^*)$, to obtain disconnection from $0$ in the original picture, and from $\infty$, to have disconnection from $\infty$ in the original picture (since $\varphi(\infty)\not\in B(0,r_0)$).

 Let $E$ be the event that the Loewner driving function of $\ksixhp$ stays within distance $r_0/100$ of $0$ up until the hitting time of $B(0,r_0/2)$. By Lemma~\ref{lem:bes-cont}, the probability of $E$ is continuous w.r.t.\ $\kappa'$, $\rho$, and the location $\varphi(x^*)$ of the force point. Thus, in the case $\varphi(x^*)\in[-2,2]$ the probability of $E$ has a positive lower bound depending only on $\nu$. When $\varphi(x^*)\not\in[-2,2]$ we can use Lemma~\ref{lem:flow_line_abs_cont} to reduce to the case of an $\SLE_{\kappa'}$ process stopped at the exit time of $B(0,r_0)$ and also obtain a uniformly positive lower bound on the probability of $E$ depending only on $\nu$. On the event $E$ we evolve $ \ksixhp$ up until the hitting time of $\partial B(0,r_0/2)$ and we consider the conformal map $\wt\varphi$ from the unbounded component of $\h \setminus \ksixhp([0,\timeksixhp{r_0/2}])$ to $\h$ sending $\ksixhp(\timeksixhp{r_0/2})$ to $0$ and fixing $i{r_0}$. By distortion estimates for conformal maps there exists $r_1\in(0,1)$ such that $\wt\varphi^{-1}(\h \cap B(0,r_1)) \subseteq B(0,r_0)$. Letting $\phi=\wt\varphi\circ\varphi$ we have that the point $\phi(x^*)$ is of distance of order~$r_0$ away from the origin and by possibly decreasing~$r_1$ we can assume that $\phi(x^*)$ is outside of $B(0,2r_1)$. Thus by Lemma~\ref{lem:flow_line_abs_cont} it suffices to consider the case of an $\SLE_{\kappa'}$ from $0$ to $\infty$ stopped at the exit time of $B(0,r_1)$. Considering for example disconnection probabilities for Brownian motion, together with the estimates in Section~\ref{sec:distortion-bounds}, we see that for any $z_1 \in B(0,r_1) \cap \h$ and $c \in (0,1)$ so that $B(z_1,c r_1) \subseteq B(0,r_1) \cap \h$ there exists some $\delta > 0$ (independently of $\epsilon > 0$) such that $\phi(B(\phi^{-1}(z_1),\delta\ep))\subseteq B(z_1,r_1c)$. Moreover, the fact that $\phi(x^*)\not\in B(0,2r_1)$ implies in particular that it suffices for $\ksixhp$ to disconnect a ball from $\infty$. By scale invariance, it thus suffices to show that an $\SLE_{\kappa'}$ from $0$ to $\infty$, that we denote by $\slehp$, disconnects a macroscopic bubble at positive distance from $\partial\h$ before exiting $\D$ with uniformly positive probability. This, in turn, is the content of the following lemma.

\begin{figure}[ht!]
			\begin{center}
				\includegraphics[scale=0.8]{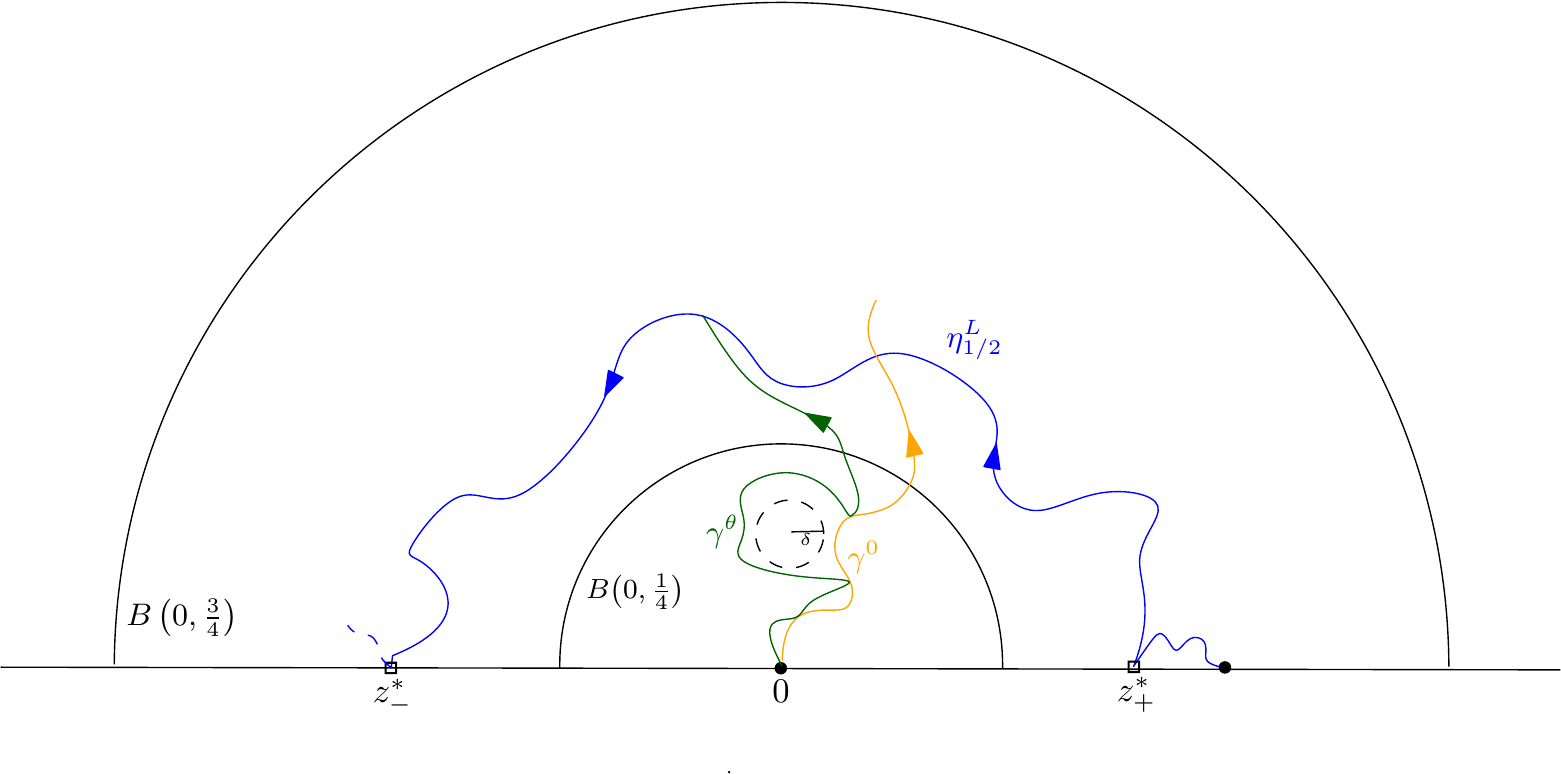}
				\caption {\label{fig:ball1-setup}An illustration of the setting used in the proof of Lemma~\ref{lem:ball-1}. For the counter flow line $\rev{\slehp}$ from $\infty$ to $0$ we consider its left boundary $\eta^L_{1/2}$ when $\rev{\slehp}$ is seen as targeting $1/2$ (and then continued towards $1/2$ as an $\SLE_{\kappa'}(\kappa'-6)$ with force point at $0$). 
				We consider the event $E_0$ that $\eta^L_{1/2}$ forms a pocket $\mcl P$ with $\partial\h$ that is contained in $B(0, 3/ 4)\cap\h$ and contains $ B(0,1/4 )\cap\h$. The extremes on $\partial \h$ of this pocket are denoted by $z_-^*$ and $z_+^*$. 
			The flow line $\gamma^\theta$ (resp.\ $\gamma^0$) of angle $\theta= \tfrac{1}{2} \tfrac{\pi\kappa}{4-\kappa}$ (resp.\ $0$) is indicated in dark green (resp.\ orange). The counter flow line $\rev{\slehp}$ visits all the points on $(\gamma^\theta\cup \gamma^0)\cap\mcl P$ before reaching the origin, after having hit $z_+^*$, and without escaping $\mcl P$ in the meantime. Thus, on the event that $(\gamma^\theta\cup \gamma^0)\cap\mcl P$ disconnect a bubble from $\infty$, also the time reversal ${\slehp}$ of $\rev{\slehp}$ disconnects such bubble before exiting $\D$.}
				\end{center}
			\end{figure}

\begin{lemma}
\label{lem:ball-1}
There exist a constant $\delta \in (0,1)$ such that the following is true. Fix $\nu > 0$ and $\kappa' \in [4+\nu,8)$. Let $\slehp$ be an $\SLE_{\kappa'}$ in~$\h$ from $0$ to $\infty$. There exists $ p \in (0,1)$ depending only on $\nu$ such that with probability at least $p$ the curve $\slehp([0,\timeksixhp{1} ])$ disconnects a bubble containing a ball of radius at least $\delta$.
\end{lemma}


\begin{proof}
See Figure~\ref{fig:ball1-setup} for an illustration of the setup of the proof and notation. We consider the time reversal $\rev{\slehp}$ of $\slehp$, which can be coupled as the counterflow line from $\infty$ to $0$ of a GFF $h$ on $\h$ with boundary conditions $-\lambda'+\pi \chi$ (resp.\ $\lambda'- \pi \chi$) on $\R_-$ (resp.\ $\R_+$). Let us also consider the counterflow line $\rev{\ksixhpz{1/2}}$ of $h$ from $\infty$ to $1/2$. We note that $\rev{\ksixhpz{1/2}}$ is an $\SLE_{\kappa'}(\kappa'-6)$ process from $\infty$ to $1/2$ with force point at $0$ and the curves $\rev{\slehp}$ and $\rev{\ksixhpz{1/2}}$ coincide up until the disconnection time of $1/2$ and $0$. 
Let $\eta^L_{1/2}$ be the left boundary of $\rev{\ksixhpz{1/2}}$. Then $\eta^L_{1/2}$ is an $\SLE_\kappa(3\kappa/2-4,-\kappa/2; \kappa/2-2)$ from $1/2$ to $\infty$ with force points at $0$ and at $(1/2)^\pm$. For fixed $\nu>0$ and $\kappa'\in[4+\nu,8)$ there exists $\nu'>0$ depending only on $\nu$ such that $\kappa$ and $(\rho^L_1,\rho^L_0;\rho^R_0)=(3\kappa/2-4,-\kappa/2; \kappa/2-2)$ satisfy the assumptions of Lemma~\ref{lem:uniform-far1} with $\nu'$ in place of $\nu$ (and inverting the roles of left and right). Thus there is a uniformly positive probability, depending only on $\nu$, that $\eta^L_{1/2}$ stays inside $B(0, 3/ 4)\setminus B(0,1/4 )$ up until first hitting $\R_-$. Let us denote this event by $E_0$ and condition on $E_0$ from now to the end of the proof. Let $\mathcal P$ be the connected component of $\h\setminus \eta^L_{1/2}$ with $0$ on its boundary. 
Let $z^*_+$ be the left-most point of $\eta^L_{1/2} \cap \R_+$, and let~$z^*_-$ be the point where $\eta^L_{1/2}$ first hits $\R_-$ (in particular, $z^*_\pm$ are on the boundary of $\mathcal P$). Then up until the time at which $\rev{\slehp}$ hits $z^*_+$ (that is, the disconnection time between $0$ and $1/2$), the curves $\rev{\slehp}$ and $\rev{\ksixhpz{1/2}}$ coincide. After the hitting time of $z^*_+$ the curve $\rev{\slehp}$ is contained in $\closure{\mathcal P}$.

Let $\theta=\tfrac 1 2\tfrac{\pi\kappa}{4-\kappa}$ be half of the critical angle, and 
let us consider the flow line $\gamma^0$ (resp.\ $\gamma^\theta$) from $0$ with angle $0$ (resp.\ $\theta$). By the interaction rules in \cite[Theorem~1.5]{ms2016ig1}, the flow line $\gamma^\theta$ travels to the left of~$\gamma^0$, and the flow lines $\gamma^0, \gamma^\theta$ both cross $\eta^L_{1/2}$ and do not cross back. The pocket formed by $\gamma^\theta$, $\gamma^0$, and $\eta^L_{1/2}$ is contained in the pocket $\mathcal P$ and the counter flow line $\rev{\slehp}$ visits all the points of $\gamma^\theta\cap\mcl P$ and $\gamma^0\cap\mcl P$ in reverse chronological order after the time it hits $z^*_+$ (and does not exit $\mcl P$ after this time). Hence, if $\gamma^\theta\cap\mcl P$ and $\gamma^0\cap\mcl P$ intersect each other creating a bubble, $\slehp$ disconnects all the points contained in such bubble before exiting~$\D$.

Let $\Phi$ be the conformal map from $\mathcal P$ to $ B(0,4)\cap\h$ fixing the origin, and mapping $z^*_-$ and $z^*_+$ to $-4$ and $4$, respectively. 
By Lemma~\ref{lem:flow_line_abs_cont}, when stopped upon exiting $B(0,2)$ the image of $\gamma^0$ (resp.\ $\gamma^\theta$) is absolutely continuous w.r.t.\ the law of a flow line of angle $0$, which we denote by $\wt\gamma^0$, (resp.\ a flow line of angle $\theta$, which we denote by $ \wt\gamma^\theta$) from $0$ to $\infty$ in $\h$, also stopped upon exiting $B(0,2)$; the
 Radon-Nikodym derivatives have second moment bounded by a constant depending only on $\nu$. Since we are on the event $E_0$, for every $c>0$ one can find some $\delta>0$ such that $\Phi(B(\Phi^{-1}(i),\delta))\subseteq B(i,c)$ and Lemma~\ref{lem:good-pockets} implies that there exists $p_\nu>0$ such that
 the probability that $ \wt\gamma^\theta$ and $\wt\gamma^0$ disconnect a ball of the type $B(i,c)$ from $\partial (B(0,2)\cap\h)$ before exiting $B(0,2)$ is bounded below by $p > 0$ which depends only on $\nu$. 
Thus, putting everything together, we conclude that $\slehp$ disconnects $B(\Phi^{-1}(i),\delta)$ from $0$ and $\infty$ before hitting $z^*_+$ (and hence before exiting $\D$) with uniformly positive probability depending only on $\nu$.
\end{proof}

 Thus, we have concluded the proof of Lemma~\ref{lem:ball-1} and hence the proof of Lemma~\ref{lem:ball-0}. We now deduce a uniform estimate on the probability that $\eta'|_{[0,\infty)}$ fills a macroscopic ball before traveling a positive distance.

 \begin{lemma}
 \label{lem:bubble-hitting-0}
 Fix $\nu>0$ and let $\kappa'\in[4+\nu,8)$. 
Let $\eta'$ be a whole-plane space-filling $\SLE_{\kappa'}$ from~$\infty$ to~$\infty$ with $\eta'(0)=0$.
For $\mathfrak{r} \geq 0$ let $T_\mathfrak{r}$ be the first time after $0$ at which $\eta$ exits $B(0,\mathfrak{r})$. 
For $\ep \in (0,1)$, let $E^\ep(\mathfrak{r})$ be the event that $\eta'([0 , T_\mathfrak{r}])$ contains a Euclidean ball of radius at least $\ep \mathfrak{r}$. 
There are constants $a_0,a_1>0$ depending only on $\nu$ such that for each $\mathfrak{r} > 0$ and each $\ep \in (0,1)$ we have that
\[
\p\!\left[E^\ep(\mathfrak{r})^c \right]\leq a_0 e^{-a_1/\ep} .
\] 
\end{lemma}

\begin{proof}

 For ease of notation we will prove the statement for the time reversal of $\eta'|_{(-\infty,0]}$ instead of for $\eta'|_{[0,\infty)}$. The statement of the lemma will then follow by the reversibility of the whole-plane space filling curve. 
Let $\rev E^\ep(\mathfrak{r})$ be the event that the time reversal $\rev{(\eta'|_{(-\infty,0]})}$ of $\eta'|_{(-\infty,0]}$ fills a ball of radius $\ep \mathfrak{r}$ before exiting $B(0,\mathfrak{r})$.
 The part of $\eta'$ given by $\eta'|_{(-\infty,0]}$ visits points in the same order as a whole-plane $\SLE_{\kappa'}(\kappa'-6)$ curve $\ksixwp$ from $\infty$ to $0$. Every time $\ksixwp$ disconnects a bubble from~$0$ which is between the left and right boundaries of $\ksixwp$ (viewed as a path in the universal cover of $\C \setminus \{0\}$), $\eta'|_{(-\infty,0]}$ fills the bubble before continuing the exploration towards $0$. Let $\rev \ksixwp$ be the time reversal of $\ksixwp$. Then $\rev \ksixwp$ is a whole-plane $\SLE_{\kappa'}(\kappa'-6)$ from $0$ to $\infty$.
Let $\rev E_{0}^\ep(\mathfrak{r})$ be the event that $(\rev \ksixwp)|_{[\timeksixwp{\ep\mathfrak{r}}, \timeksixwp{\mathfrak{r}}]}$ disconnects a ball of radius at least $\ep \mathfrak{r}$ (where here $\timeksixwp{\mathfrak{r}}$ is as in~\eqref{eqn:s_r_def} for $\rev \ksixwp$). Recall that by this we mean that there exists $\timeksixwp{\ep\mathfrak{r}} \leq s < t \leq \timeksixwp{\mathfrak{r}}$ such that $(\rev \ksixwp)|_{[s,t]}$ does not disconnect $0$ from $\infty$ and the set given by the closure of the union of $\rev \ksixwp ([s,t])$ with the bounded connected components of $\C\setminus \rev\ksixwp([s,t])$ that do not contain~$0$ contains a ball of radius $\ep \mathfrak{r}$. We will now argue that $\rev E_{0}^\ep(\mathfrak{r})\subseteq\rev E^\ep(\mathfrak{r})$. Indeed, on the event $\rev E_{0}^\ep(\mathfrak{r})$ we have that $\ksixwp$ disconnects a ball of radius $\ep \mathfrak{r}$, contained in $B(0,\mathfrak r)$, from $0$ and from $\partial B(0,\mathfrak{r})$ after hitting $B(0,\mathfrak{r})^\mathrm {c}$ for the last time. Thus the curve $\eta'|_{(-\infty,\infty)}$ fills this bubble after hitting $B(0,\mathfrak{r})^c$ for the last time before hitting $0$. Since the bubble is also disconnected from~$0$ by $\ksixwp$, we have that $\eta'$ finishes filling the bubble before hitting $0$. That, is $\rev{(\eta'|_{(-\infty,0]})}$ fills a ball of radius $\ep \mathfrak{r}$ before exiting $B(0,\mathfrak{r})$.
 
Fix $C >1$ to be chosen later. For $\ep \in (0,1)$, let $N_\ep := \lfloor (C \ep)^{-1} \rfloor$ and for $1 \leq k \leq N_\ep$, let $\mathfrak{r}^\ep(k) :=k C \ep \mathfrak{r}$. For $2 \leq k \leq N_\ep$, let $\rev F_{0}^\ep(k)$ be the event that $\rev\ksixwp([\timeksixwp{\mathfrak{r}^\ep(k-1)} ,\timeksixwp{\mathfrak{r}^\ep(k )} ])$ disconnects a ball of radius at least $\ep \mathfrak{r}$. 
Then $\bigcup_{k=2}^{N_\ep} \rev F_{0}^\ep(k) \subseteq \rev E_{0}^\ep(\mathfrak{r})$. By Lemma~\ref{lem:ball-0} and the scale invariance of the law of whole-plane SLE$_{\kappa'}(\kappa'-6)$, we can choose $C>1$ sufficiently large and $p \in (0,1)$ sufficiently small, depending only on $\nu$, so that for $1 \leq k \leq N_\ep-1$, 
\[ 
\p\!\left[ \rev F_{0}^\ep(k+1) \giv \rev\ksixwp|_{\left(-\infty ,\timeksixwp{\mathfrak{r}^\ep(k )}\right]} \right] \geq p \quad \text{ a.s.} 
\]
Therefore for an appropriate choice of $a_0 , a_1 >0$ we have that
\[
\p\!\left[ \rev E_0^\ep(\mathfrak{r})^c \right]\leq (1-p)^{N_\ep-1} \leq a_0 e^{-a_1/\ep },
\]
as in the statement of the lemma. This implies that $\p\!\left[ \rev E^\ep(\mathfrak{r})^c \right]\leq (1-p)^{N_\ep-1} \leq a_0 e^{-a_1/\ep }$, and, finally, since $\rev{(\eta'|_{(-\infty,0]})}$ has the same law of $\eta'|_{[0,\infty)}$, we conclude that $\p\!\left[E^\ep(\mathfrak{r})^c \right]\leq a_0 e^{-a_1/\ep}$.
\end{proof}

 We now turn to show that, with large probability, $\eta'([-T,T])$ for $T > 0$ large creates a large pocket around the origin.
 
 We consider the coupling between $\eta'$ and a whole-plane GFF $h$, defined modulo a global additive multiple of $2\pi\chi$, as in~\cite{ms2017ig4}. For $z\in\C$, let $\eta_z^L$, $\eta_z^R$ be the flow lines of $h$ started from $z$ with angles $+\pi/2$ and $-\pi/2$, respectively. Then these flow lines $\eta_z^L$, $\eta_z^R$ form the left and right boundaries of $\eta'$ stopped at the first time it hits~$z$ by the construction of $\eta'$ given in \cite{ms2017ig4}.

For $\ep > 0$ and $\mathfrak{r} > 0$, let 
\[ 
\mcl S_\ep(\mathfrak{r}) := B(0,\mathfrak{r}) \cap (\ep \Z^2) .
\] 

 \begin{lemma}
 \label{lem:large-pocket}
 Fix $\nu > 0$ and $\kappa' \in [4+\nu,8)$. Fix $R>0$ and for $\mathfrak{r} >R$ let $\wt F (\mathfrak{r}) = \wt F (\mathfrak{r},R)$ be the event that the following is true. For each $\ep \in (0,1)$, there exists $z_0 , w_0 \in \mcl S_\ep(\mathfrak{r}) \setminus B(0,R)$ such that the flow lines $\eta_{w_0}^L$ and $\eta_{z_0}^L$ and the flow lines $\eta_{z_0}^R$ and $\eta_{w_0}^R$ merge and form a pocket containing $B(0,R)$ before leaving $B(0,\mathfrak{r})$. Then for each fixed $R> 0$ and each $p\in(0,1)$ there exists $\mathfrak{r}=\mathfrak{r}(\nu,p)>R$ depending only on $\nu,p$ such that $\p[\wt F(\mathfrak{r})] \geq p$.
\end{lemma}

\begin{proof}Let $h$ be a whole-plane GFF defined modulo $2\pi\chi$. For $r>0$, we denote by $\CF_r$ the $\sigma$-algebra generated by $h|_{B(0,r)}$. Let $h=\wt h_r+\mathfrak h_r$ be the Markov decomposition of $h$, where $\wt h_r$ is a zero boundary GFF in $B(0,r)^c$ and $\mathfrak h_r$ is defined modulo a global additive constant in $2\pi\chi\Z$, is harmonic in $B(0,r)^c$, and coincides with $h$ inside $B(0,r)$. 

Let $h_r(0)$ be the average of $h$ on $\partial B(0,r)$ (defined modulo $2\pi\chi$).
We can consider $\mathfrak{r}\geq 2R$ and fix two points $z_0 , w_0 \in \mcl S_\ep(2R) \setminus B(0,(3/2)R)$. For every $k\in\N$ let $F_k$ be the event that~$\eta_{2^kw_0}^L$ and~$\eta_{2^kz_0}^L$ and the flow lines $\eta_{2^kz_0}^R$ and $\eta_{2^kw_0}^R$ merge and form a pocket containing $B(0,2^k(3/2)R)$ before leaving $B(0,2^{k+1}R)$. 
For every $k\in\N$ assume that $h_k^0$ is a zero boundary GFF in $\C\setminus B(0,2^kR)$, then for every fixed $\kappa'$ there exists some $q>0$ such that the probability (under the law of $h_k^0$) of $F_{k}$ is larger than $q$. 
Moreover, we can obtain a lower bound for $q$, depending only on $\nu$ (and not on $\kappa'$), using Lemma~\ref{lem:uniform-far0} and Lemma~\ref{lem:boundary-merge} (We note that we can apply Lemma~\ref{lem:uniform-far0} for a flow line starting from an interior point because once we have drawn an initial segment of the path its continuation is a radial $\SLE_\kappa(2-\kappa)$ which in turn is by \cite[Theorem~3]{sw2005coordinate} the same as a chordal $\SLE_\kappa(2-\kappa;2\kappa-8)$.) By the scale invariance of the GFF, $q$ is also independent of the scale $k$.
Thus, there exists $q=q(\nu)>0$ such that for every $k\in \N$,
\[\p[ F_{k}(h_k^0)]\geq q.\] 
By~\cite[Section 4.1]{mq2020geodesics} the conditional law of $h-h_{2^kR}(0)$ given its values in $B(0,2^kR)$ is mutually absolutely continuous w.r.t.\ the law of a zero boundary GFF $\C\setminus B(0,2^kR)$ when both laws are restricted to $B(0,2^k(3/2)R)^c$. 
Moreover, for every $M>0$, when one restricts to \emph{M-good scales}, that is, scales $k$ at which
\[\sup_{w\in B(0,2^k(3/2)R)^c}|\mathfrak h_{2^kR}(w)- h_{2^kR}(0)|<M
\]
it is also possible to bound the second moments for the Radon-Nikodym derivatives by some constant $c(M)\in(0,\infty)$~\cite[Lemma 4.1]{mq2020geodesics}. 
 Therefore, by the Cauchy-Schwarz inequality, as in~\cite[Remark~4.2]{mq2020geodesics} we have that at every $M$-good scale $k$, 
 \[\p[ F_{k}(h)\giv \CF_{2^kR}]\geq \left( \frac 1{c(M)}\p[ F_{k}(h_k^0)]\right)^{2}\geq \left( \frac {q(\nu)}{c(M)}\right)^2.\]
 Fix $K\in\N$ and let $N(K,M)$ be the number of $M$-good scales $1\leq k\leq M$. \cite[Lemma~4.3]{mq2020geodesics} shows in particular that, for every fixed $K$, and for every $a>0$ and $\beta\in(0,1)$ there exists $M(a,\beta)$ and $c_0(a,\beta)$ such that 
\[\p[N(K,M)\leq \beta K]\leq c_0(a,\beta)e^{-aK}\]
By the definition of $M$-good scales, we see that the value $M(a,\beta)$ is independent of $\kappa'$. Thus, 
\begin{align*}
\p\!\left[\wt F\left(2^KR\right)^c\right]&\leq \p[\cap_{1\leq k\leq K} F_k^c(h), N(K,M)>\beta K] + \p[N(K,M)\leq \beta K]\\
&\leq \left(1-\left( \frac {q(\nu)}{c(M)}\right)^2 \right)^{\beta K}+c_0(a,\beta)e^{-aK}.
\end{align*}
Thus it suffices to consider $a=1$, $\beta=1/2$, $M=M(1,\tfrac 1 2)$, and choose $K=K(\nu,p)$ large enough so that the bound above is smaller than $1-p$.
\end{proof}

\subsection{Pocket diameter}
\label{subsec:pocket_diameter}

In this section we show uniform bounds on the size of the pockets formed by right and left boundaries flow lines started from points in $\mcl S_\ep = \epsilon \Z^2\cap [-2,2]^2$. The main result of this section is Lemma~\ref{lem:4.14-unif} and the rest of the paper can be read independently from the rest of this section. To prove Lemma~\ref{lem:4.14-unif}, we repeat the arguments from~\cite[Section~4.3.1]{ms2017ig4} applying the lemmas we derived in Section~\ref{sec:uniform_flow_line_behavior} instead of their $\kappa$-dependent analogs, and see that uniformity in $\kappa$ carries out throughout the proof.

Let $h$ be a whole-plane GFF viewed as a distribution defined up to a global multiple of $2\pi \chi$. For each $z \in \C$, let $\eta_z$ be the flow line of $h$ starting from $z$. Fix $K \geq 5$ (which we will take to be large but independent of $\epsilon$ and $\kappa$). Fix $z_0 \in [-1,1]^2$ and let $\eta = \eta_{z_0}$. For each $z \in \C$ and $n \in \N$, let $\tau_z^n$ be the first time at which $\eta_z$ exits $B(z_0,K n\epsilon)$ and let $\wh{\tau}_z^n$ be the first time that $\eta_z$ exits $B(z_0,(K n+1/2)\epsilon)$. Note that $\tau_z^n = 0$ if $z \notin B(z_0,K n \epsilon)$. Let $\tau^n = \tau_{z_0}^n$, $\wh{\tau}^n = \wh{\tau}_{z_0}^n$.

\begin{lemma}
\label{lem:4.14-unif}
There exists a constant $K_0 \geq 5$ such that the following holds. Fix $\nu > 0$ and $\kappa \in [\nu,4-\nu]$.
 There exists a constant $C > 0$ depending only on $\nu$ such that, for every $K \geq K_0$ and $n \in \N$ with $n \leq (K\epsilon)^{-1}$, the probability that $\eta|_{[0,\tau^n]}$ does not merge with any of $\eta _z|_{[0,\tau_z^n]}$ for $z \in \mcl S_\ep \cap B(z_0,K n \epsilon)$ is at most $e^{-C n}$.
\end{lemma}

For each $n \in \N$, let $w_n \in \mcl S_\ep \setminus B(z_0,(K n+1) \epsilon )$ be such that $|\eta({\tau}^n) - w_n| \leq 2\epsilon$ (where ties are broken according to some fixed rule). Let $\theta_0 = \tfrac 1 2 \tfrac{\pi\kappa}{4-\kappa}$ be half of the critical angle. A flow line of angle $\theta_0$ a.s.\ hits a flow line of zero angle started at the same point on its left side~\cite[Theorem~1.7]{ms2017ig4}. For each $n \in \N$, we let $\gamma_n $ be the flow line of $h$ starting from $\eta(\wh{\tau}^n)$ with angle $\theta_0$ and let $\sigma^n$ be the first time that $\gamma_n$ leaves $B(z_0,(K n+4)\epsilon)$.
 We consider the $\sigma$-algebra $\CF_n$ generated by $\eta|_{[0,\tau^n]}$ and by $\gamma_{i}|_{[0,\sigma^{i}]}$ for $j=1,\ldots,{n-1}$. Let $A_n = \eta([0,\tau^{n+1}]) \cup \gamma_n([0,\sigma^n])$ and let $P_n$ be the connected component of $\C \setminus A_n$ which contains $w_n$. We consider the event $E_n$ that the set $A_n$ separates $w_n$ from $\infty$ and that the harmonic measure of the left side of $\eta([0,\tau^{n+1}])$ as seen from $w_n$ inside $P_n$ is at least~$\tfrac{1}{4}$. 
\begin{lemma}
\label{lem:good-pockets-1}
There exists constants $K_0 \geq 5$ such that the following holds. Fix $\nu > 0$ and $\kappa \in [\nu,4-\nu]$, there exists $p_1 > 0$ depending only on $\nu$ such that for every $n \in \N$ and $K \geq K_0$ we have
\[ \p[ E_n \giv \CF_{n-1}] \geq p_1.\] 
\end{lemma}
\begin{proof}
Let $\psi $ be the conformal map from the unbounded connected component $U$ of $\C \setminus (\eta ([0,\tau^n]) \cup \bigcup_{j=1}^{{n-1}} \gamma_{j}([0,\sigma^{j}]))$ to $\h$ with $\psi(\eta(\tau^n)) = 0$ and $\psi(w_n) = i$.
By the conformal invariance of Brownian motion and the Beurling estimate~\cite[Theorem~3.69]{lawler2005conformally} it follows that there exists a constant $K \geq 5$ sufficiently large such that the images of $\cup_{j=1}^{{n-1}} \gamma_{j}([0,\sigma^{j}])$ and $\infty$ under $\psi$ lie outside of $B(0,100)$. Thus, the law of $\wt{h} = h \circ \psi^{-1} - \chi \arg (\psi^{-1})'$ restricted to $B(0,50)$ is mutually absolutely continuous with respect to the law of a GFF on $\h$ restricted to $B(0,50)$ whose boundary data is chosen so that its flow line starting from $0$ is a chordal $\SLE_\kappa(2-\kappa)$ process having force point $x^*$ at the image under $\psi$ of the most recent intersection of $\eta|_{[0,\tau^n]}$ with itself (or just a chordal $\SLE_\kappa$ process if there is no such self-intersection point which lies in $\psi^{-1}(B(0,50))$). By Lemma~\ref{lem:flow_line_abs_cont} (and the Cauchy-Schwarz inequality), we can therefore consider the case of an $\SLE_\kappa(2-\kappa)$ process $\eta$.
 By the distortion estimate in Lemma~\ref{lem:ball_bound} there exists some constant $c$ independent of $\ep$ such that $B(i,c)\subseteq \psi(B(w_n,\ep/4))$. We consider the event $E^0_n$ that $\eta$ stays closer than $c/10$ to the line segment $[0,i]$ up until hitting $\psi(\partial B(z_0,(K n+1/2) \epsilon))$. The probability of this event is bounded below by the probability that $\eta$ stays closer than $c/10$ to the line segment $[0,i]$ up until hitting $B(i,c)$, and this probability can be controlled in a way that only depends on $\nu$. Indeed, let us consider $\delta$ small compared to $c$. By the Markov property of $\SLE_\kappa(\rho)$ processes and continuity of the driving process, there exists $\delta$ small enough that the probability in the case $x^*\in[-\delta,\delta]$ is bounded below by the probability corresponding to $x^*=0_+$ or $x^*=0_-$. Thus, fixed such $\delta$, the claim in the case $x^*\in[-\delta,\delta]$ follows by Lemma~\ref{lem:uniform-far0}. If $x^*\notin[-\delta,\delta]$ then we can use the usual argument from Lemma~\ref{lem:flow_line_abs_cont} to reduce to the case of an ordinary $\SLE_\kappa$ process. On the event $E^0_n$ we evolve $\eta$ up until the hitting time of $\psi(\partial B(z_0,(K n+1/2) \epsilon))$ and map back to the upper half-plane fixing $i$ and sending the tip to the origin. The image of $x^*$ is now at a macroscopic distance from the origin, and by Lemma~\ref{lem:flow_line_abs_cont}, we can therefore consider the setting of Lemma~\ref{lem:good-pockets} and obtain a uniform lower bound.
\end{proof}

On $E_n$, let $F_n$ be the event that $\eta_{w_n}$ merges with $\eta$ upon exiting $P_n$. Let $\CF = \sigma(\CF_n : n \in \N)$.

\begin{lemma}
\label{lem:merge-good-pockets}
There exists constants $K_0 \geq 5$ such that the following holds. Fix $\nu > 0$ and $\kappa \in [\nu,4-\nu]$, there exists $p_2> 0$ depending only on $\nu$ such that for every $n \in \N$ and $K \geq K_0$ we have
\[ \p[ F_n \giv \CF] \one_{E_n} \geq p_2 \one_{E_n}.\]
\end{lemma}
\begin{proof}
This follows from Lemma~\ref{lem:boundary-merge} and the absolute continuity argument in Lemma~\ref{lem:flow_line_abs_cont}.
\end{proof}

\begin{proof}[Proof of Lemma~\ref{lem:4.14-unif}]
We choose $K_0$ large enough so that Lemma~\ref{lem:good-pockets-1} and Lemma~\ref{lem:merge-good-pockets} both apply. From Lemma~\ref{lem:good-pockets-1} it follows that it is exponentially unlikely that we have fewer than $\tfrac{1}{2} p_1 n$ of the events $E_j$ occur and by Lemma~\ref{lem:merge-good-pockets} it is exponentially unlikely that we have fewer than a $\tfrac{1}{2} p_2$ fraction of these in which $F_j$ occurs.
\end{proof}

\section{Compactness:  proof of Theorem~\ref{thm:compactness}}
\label{sec:tightness}

Fix $\nu>0$. The purpose of this section is to deduce the tightness w.r.t.\ $\kappa'\in [4+\nu,8)$ of the laws of space-filling $\SLE_{\kappa'}$ curves. This will follow mainly from the fact that it is uniformly likely in $\kappa'$ that a space-filling $\SLE_{\kappa'}$ fills in a macroscopic ball before traveling a long distance. 

\subsection{Tightness of the laws w.r.t.\ $\kappa'$} 

We start by stating the main result from this section, which implies Theorem~\ref{thm:compactness}.

\begin{proposition}
\label{prop-sle-tight}
Fix $\nu>0$ and let $\kappa'\in[4+\nu,8)$. Let $ \eta'$ be space-filling $\SLE_{\kappa'}$ from $\infty$ to $\infty$ in~$\C$ parameterized by Lebesgue measure. For each $T>0$, $r\in(0,1)$, and $\wt\delta>0$, let
 $K_{\wt\delta}=K_{T,r,\wt\delta}$ be the set of functions 
\[
K_{\wt\delta}:=\{f \colon [0,T] \to \C : \forall\, |t-s|\leq\wt\delta,\,|f(t)-f(s)|<( \tfrac 2 \pi |t-s|)^{(1-r)/2}\}.
\]
Then for every $\xi>0$ there exists $\wt\delta=\wt\delta(\nu,\xi,r)>0$, depending only on $\nu$, $\xi$, and $r$, such that 
\[
 \p[\eta' \in K_{\wt\delta}]\geq 1-\xi.
\]
In particular, we have that the law of $\eta'|_{[0,T]}$ is tight in $\kappa'\in[4+\nu,8 )$ with respect to the uniform topology.
\end{proposition}

The main input in the proof of Proposition~\ref{prop-sle-tight} is the following result, which corresponds to~\cite[Proposition~3.4]{ghm2020kpz}, with the difference that we give a statement which holds uniformly in $\kappa'$.

 \begin{proposition}
 \label{prop:bubble}
 Fix $\nu>0$ and let $\kappa'\in[4+\nu,8)$. Let $\eta'$ be a space-filling $\SLE_{\kappa'}$ from $\infty$ to $\infty$ in~$\C$. For $r \in (0,1)$, $R > 0$, and $ \ep> 0$, let $\mcl E_{ \ep} = \mcl E_{ \ep}(R,r)$ be the event that the following is true. For each $\delta\in (0, \ep]$ and each $a < b\in\R$ such that $\eta'([a,b]) \subseteq B(0,R)$ and $\op{diam} \eta'([a,b]) \geq \delta^{1-r}$, the set $\eta'([a,b])$ contains a ball of radius at least $\delta $. 
 Then for each $q>0$ there exists $ \ep= \ep(\nu,q,R,r)$ such that 
\[
 \p\!\left[\mcl E_\ep \right] \geq 1-q .
\]

\end{proposition} 

Before we give the proof of Proposition~\ref{prop:bubble}, we will first explain why it implies Proposition~\ref{prop-sle-tight}.
 
\begin{proof}[Proof of Proposition~\ref{prop-sle-tight}] We first argue that there exists $R=R(\nu,\xi, T) > 0$ so that $\p[\eta' ([0,T])\subseteq B(0,R)]\geq 1-\xi/2$.

Let $T_\mathfrak{r} = \inf\{t \geq 0 : \eta'(t) \notin B(0,\mathfrak r)\}$ and let $E^{\ep_0}(\mathfrak r)$ be the event that $\eta'([0,T_\mathfrak{r}])$ fills a ball of radius $\ep_0\mathfrak r$. By Lemma~\ref{lem:bubble-hitting-0} we can choose $\ep_0 > 0$ so that $\p[E^{\ep_0}(\mathfrak r)] \geq 1-\xi/2$ for every $\mathfrak r>0$. Let us fix $R=\sqrt{2T/(\pi\ep_0^2)}$. Then on $E^{\ep_0}(R)$ we have that $T_R\geq 2T$ since $\eta'$ is parameterized by area, thus $\eta'([0,T])\subseteq B(0,R)$. 

Let us choose $\delta_0=\delta_0(\nu,\xi/2,R,r) > 0$ from Proposition~\ref{prop:bubble} so that $\p[\mcl E_{\delta_0}] \geq 1-\xi/2$. On the event $\mcl E_{\delta_0}$, for all $\delta \in (0,\delta_0)$ and each $a < b\in\R$ such that $\eta'([a,b]) \subseteq B(0,R)$ and $\op{diam} \eta'([a,b]) \geq \delta^{1-r}$, the set $\eta'([a,b])$ contains a ball of radius at least $\delta$, hence $b-a \geq \pi\delta^2$. Hence for all $\ep \in (0,\delta_0^2)$, it holds that whenever $a<b$ are such that $\eta'([a,b])\subseteq B(0,R)$ and $b-a < \pi \ep$, we have $\op{diam} \eta'([a,b]) \leq \ep^{(1-r)/2}$. 
Thus
\[
| \eta'(a)-\eta'(b)| \leq (\tfrac 2 \pi |b-a|)^{(1-r)/2}\quad\forall \,a,b \in [0,T] \quad \text{with} \quad |a-b|\leq \delta_0^2\pi.
\]
That is, $\eta' \in K_{\delta_0^2\pi}$ hence $\p[\eta' \in K_{\delta_0^2\pi}] \geq 1-\xi$.\end{proof}

\subsection{Proof of Proposition~\ref{prop:bubble}}
In order to prove the result, we will follow the strategy from~\cite[Section 3.2]{ghm2020kpz} keeping track of the dependencies in $\kappa'$.

From Lemma~\ref{lem:bubble-hitting-0} and translation invariance of the two sided whole-plane space-filling $\SLE_{\kappa'}$, we immediately obtain the following.
 \begin{lemma}
 \label{lem:bubble-hitting}
Fix $\nu>0$ and let $\kappa'\in[4+\nu,8)$. There are constants $a_0,a_1>0$ depending only on $\nu$ such that the following is true. Let $\eta'$ be a whole-plane space-filling $\SLE_{\kappa'}$ from $\infty$ to $\infty$ with any choice of parameterization.
For $z\in\C$, let $\tau_z$ be the first time $\eta'$ hits $z$ and for $\mathfrak{r} \geq 0$ let $\tau_z(\mathfrak{r})$ be the first time after $\tau_z$ at which $\eta$ exits $B(z,\mathfrak{r})$. 
For $\ep \in (0,1)$, let $E_z^\ep(\mathfrak{r})$ be the event that $\eta'([\tau_z , \tau_z(\mathfrak{r})])$ contains a Euclidean ball of radius at least $\ep \mathfrak{r}$. For each $\mathfrak{r} > 0$ and $\ep \in (0,1)$ we have that
\[
\p\!\left[E_z^\ep(\mathfrak{r})^c \right]\leq a_0 e^{-a_1/\ep} .
\] 
\end{lemma}

For $\ep > 0$ and $\mathfrak{r} > 0$, let 
\[ 
\mcl S_\ep(\mathfrak{r}) := B(0,\mathfrak{r}) \cap (\ep \Z^2) .
\] 

By Lemma~\ref{lem:bubble-hitting} and a union bound we obtain that with large probability each segment of the space-filling $\SLE_{\kappa'}$ curve $\eta'$, which starts from the first time $\eta'$ hits a point of $\mcl S_\ep(\mathfrak{r})$ and having diameter at least $\ep^{1-r}$, contains a ball of radius at least $\ep$. To complete the proof we show that there are no segments of the space-filling $\SLE_{\kappa'}$ that have diameter at least $\ep^{1-r}$ and do not encounter points in $\mcl S_\ep(\mathfrak{r})$. We consider the coupling between the space-filling $\SLE_{\kappa'}$ curve $\eta'$ and a whole-plane GFF $h$, defined modulo a global additive multiple of $2\pi\chi$ as in~\cite{ms2017ig4}.

For each $z\in\C$ we consider the curve $\eta'$ stopped at the first time it hits $z$, the left and right boundaries this curve are formed by the flow lines of $h$ started from $z$ and having angles $\pi/2$ and $-\pi/2$~\cite{ms2017ig4}. We denote these flow lines by $\eta_z^L$ and $\eta_z^R$, respectively.

\begin{lemma}
\label{prop:grid-pockets}
Suppose we are in the setting described above. Fix $\nu>0$ and let $\kappa'\in[4+\nu,8)$. Fix $R>0$ and, for $\ep \in (0,1)$ and $\mathfrak{r} > R$, let $\mcl P_\ep(\mathfrak{r})$ be the set of complementary connected components of 
\[
\bigcup_{z\in \mcl S_\ep(\mathfrak{r})} (\eta_z^R \cup \eta_z^L) 
\]
which intersect $B(0,R)$.
For $r \in (0,1)$, let $\wt{\mcl E}_\ep(\mathfrak{r}) = \wt{\mcl E}_\ep(\mathfrak{r} ;R,r)$ be the event that $\sup_{P\in\mcl P_\ep(\mathfrak{r})} \op{diam} P \leq \ep^{1-r}$. Also let $\wt F (\mathfrak{r})$ be defined as in Lemma~\ref{lem:large-pocket}. Then for each fixed $\mathfrak{r} > R$, we have
\[ 
 \p\!\left[\wt{\mcl E}_\ep(\mathfrak{r})^c \cap \wt F (\mathfrak{r}) \right] \leq \mathfrak{r}^2 o_\ep^\infty(\ep) 
\]
where by $o_\ep^\infty(\ep)$ we mean that the decay to $0$ is faster than any polynomial of $\ep$ and the implicit constants depend only on $R$ and $\nu$. 
\end{lemma}
\begin{proof}

For $\mathfrak{r} > R$ and $z\in B(0,\mathfrak{r})$, let $\wt E_\ep^z(\mathfrak{r}) $ be the following event. There exists $w \in \mcl S_\ep(\mathfrak{r})$ such that $w\not=z$ and the curve $\eta_z^L$ (resp.\ $\eta_z^R$) hits and subsequently merges with $\eta_w^L$ (resp.\ $\eta_z^R$) on its left (resp.\ right) side before leaving $\partial B(z,\ep^{1-r}/8)$. Let
\[
\wt{\mcl E}_\ep'(\mathfrak{r}) := \bigcap_{z\in \mcl S_\ep(\mathfrak{r})} \wt E_\ep^z(\mathfrak{r}) . 
\]

 Lemma~\ref{lem:4.14-unif} implies that $\p\!\left[\wt E_{\ep,\kappa}^z(\mathfrak{r})^c \right] = o_\ep^\infty(\ep)$ at a rate depending only on~$\nu$. Thus, by a union bound over $z \in \mcl S_\ep(\mathfrak{r})$, we have that $\p\!\left[\wt{\mcl E}_{\ep,\kappa}'(\mathfrak{r})\right] = 1-\mathfrak{r}^2 o_\ep^\infty(\ep)$ where the implicit constants depend only on~$\nu$. Thus to complete the proof it suffices to prove that $\wt{\mcl E}_\ep'(\mathfrak{r}) \cap \wt F(\mathfrak{r}) \subseteq \wt{\mcl E}_\ep(\mathfrak{r})$.

We start by showing that, on the event $\wt F(\mathfrak{r})$, the boundary of every pocket $P\in\mcl P_\ep(\mathfrak{r})$ is completely traced by curves $\eta_z^L$, $\eta_z^R$ for $z\in \mcl S_\ep(\mathfrak{r})$. To this end, let $z_0, w_0 \in \mcl S_\ep(\mathfrak{r})$ be points as in the definition of $\wt F (\mathfrak{r})$ in Lemma~\ref{lem:large-pocket}. We consider the pocket $P_0$ formed by $\eta_{z_0}^L$, $\eta_{z_0}^R$, $\eta_{w_0}^L$, and $\eta_{w_0}^R$, which surrounds $B(0,R)$. Then for any $w\notin B(0,\mathfrak r)$, the flow lines $\eta_w^q$ for $q\in\{L,R\}$ cannot enter $B(0,R)$ without merging into flow lines $\eta_z^L$, $\eta_z^R$ for $z\in \mcl S_\ep(\mathfrak{r})$ (as they cannot cross $\partial P_0$ without merging into $\eta_{z_0}^q$ or $\eta_{w_0}^q$).

Now assume the event $\wt{\mcl E}'_\ep(\mathfrak{r}) \cap \wt F(\mathfrak{r})$ occurs and let $P\in \mcl P_\ep(\mathfrak{r})$. We proceed to argue that $\op{diam} P \leq \ep^{1-r}$. The occurrence of $\wt F (\mathfrak{r})$ implies that $\partial P$ is formed by either four arcs formed by the flow lines $\eta_z^q$ and $\eta_w^q$ for $q\in\{L,R\}$ for some $z,w \in \mcl S_\ep(\mathfrak{r})$; or by only two arcs formed by $\eta_z^L$ and $\eta_z^R$ for some $z\in \mcl S_\ep(\mathfrak{r})$.

We first assume that we are in the former case, that is, there exist $z,w\in \mcl S_\ep(\mathfrak{r})$ such that $\partial P$ is formed by non-trivial arcs traced by the right side of $\eta_z^R$, the left side of $\eta_z^L$, the right side of $\eta_w^R$, and the right side of $\eta_w^L$. Consider the arc $I_z^L$ of $\partial P$ formed by the left side of $\eta_z^L$. We note that there cannot exist any $v\in \mcl S_\ep(\mathfrak{r})$ such that the curve $\eta_z^L$ hits $\eta_v^L$ on its left side before $\eta_z^L$ finishes tracing $I_z^L$ (since otherwise $I_z^L$ would partially lie on the boundary of a pocket different from $P$). The same also holds replacing $\eta_z^L$ with any of the other three arcs forming $\partial P$. Each of the arcs forming $\partial P$ has diameter at most $\frac14 \ep^{1-r}$, since we are on the event $\wt E_\ep^z(\mathfrak{r}) \cap \wt E_\ep^w(\mathfrak{r})$. Thus, we obtain $\op{diam} P \leq \ep^{1-r}$. When $\partial P$ is traced by only two flow lines $\eta_z^L$ and $\eta_z^R$ for $z\in \mcl S_\ep(\mathfrak{r})$, instead of four, a similar argument considering only two distinguished boundary arcs implies that $\op{diam} P \leq \ep^{1-r}$ holds also in this case. Thus we obtain $\wt{\mcl E}_\ep'(\mathfrak{r}) \cap \wt{F}(\mathfrak{r}) \subseteq \wt{\mcl E}_\ep(\mathfrak{r})$, completing the proof of the lemma.\end{proof}

\begin{proof}[Proof of Proposition~\ref{prop:bubble}]
We start by fixing $r' \in (0,r)$ and, for $\ep \in (0,1)$, we set $k_\ep$ to be the largest $k\in\N$ for which $2^{-k} \geq \ep$. Let $ \mcl E_\ep^1$ be the event that the following occurs. For each $k \geq k_\ep$ and each $z \in \mcl S_{2^{-k}}(2R)$, the event $E_z^{2^{-kr'} }(2^{-k( 1-r')} )$ of Lemma~\ref{lem:bubble-hitting} holds (using $2^{-k( 1-r')}$ in place of $\mathfrak{r}$ and $2^{-kr'}$ in place of $\ep$). Then Lemma~\ref{lem:bubble-hitting} together with a union bound implies
\[
\p\!\left[\mcl E_\ep^1 \right] = 1-o_\ep^\infty(\ep),
\]
at a rate depending only on $\nu$.
Fix $\mathfrak{r} > R$ that will be chosen later in the proof, let $\mcl P_\ep(\mathfrak{r})$ be the set of pockets from Lemma~\ref{prop:grid-pockets} and let $\wt F(\mathfrak{r})$ be the event from Lemma~\ref{lem:large-pocket}. 
 We also let $\wt{\mcl E}_\ep(\mathfrak{r})$ be the event defined in Lemma~\ref{prop:grid-pockets} using $r'$ in place of $r$, and set
\[
 \mcl E_\ep^2(\mathfrak{r}) := \bigcup_{k=k_\ep}^\infty \wt{\mcl E}_{2^{-k}}(\mathfrak{r}) .
\]
Combining Lemma~\ref{prop:grid-pockets} with a union bound and the argument given just above, we obtain
\[
\p\!\left[ \mcl E_\ep^1 \cap \mcl E_\ep^2(\mathfrak{r}) \cap \wt F(\mathfrak{r}) \right] = \p\!\left[\wt F(\mathfrak{r})\right]- \mathfrak{r}^2 o_\ep^\infty(\ep),
\] at a rate depending only on $R$ and $\nu$. Lemma~\ref{lem:large-pocket} implies that $\p\!\left[\wt F(\mathfrak{r})\right] \rta 1$ as $\mathfrak{r} \rta \infty$ at a rate that depends only on $\nu$. Thus, showing that $\mcl E_\ep^1 \cap \mcl E_\ep^2(\mathfrak{r}) \subseteq \mcl E_\ep$ for each sufficiently small $\ep \in (0,1)$ and each choice of $\mathfrak{r} > R$, is enough to complete the proof of the proposition.

Suppose $\mcl E_\ep^1 \cap \mcl E_\ep^2(\mathfrak{r})$ occurs and that for some $\delta\in (0,\ep]$ and $a < b\in\R$ it holds that $\eta'([a,b]) \subseteq B(0,R)$ and $\op{diam} \eta'([a,b]) \geq \delta^{1-r}$.
Set $k_\delta$ to be the largest $k\in\N$ such that $2^{-k} \geq \delta$ and let $t_*$ be the smallest $t\in [a,b]$ such that $\op{diam} \eta'([a,t_*]) \geq \frac12 \delta^{1-r}$. 
 The fact that, replacing $r$ by $r'$, the event $\wt{\mcl E}_{2^{-k_\delta}}(\mathfrak{r})$ occurs implies that, for $\ep$ small enough, there cannot be a single pocket in $\mcl P_{2^{-k_\delta}}(\mathfrak{r})$ that entirely contains $\eta'([a,t_*])$.
At the same time, every time the curve $\eta'$ exits a pocket in $\mcl P_{2^{-k_\delta}}(\mathfrak{r})$, it necessarily hits a point $z\in \mcl S_{2^{-k_\delta}}(2 R)$ for the first time. Therefore there exist a point $z\in \mcl S_{2^{-k_\delta}}(2 R)$ and a time $t_{**} \in [a,t_*]$ such that at time $t_{**}$ the curve $\eta'$ hits $z$ for the first time. The diameter of set $\eta'([t_{**}, b])$ is at least $\frac12 \delta^{1-r}$ and, on the event $E_{2^{-k_\delta}}(z)$, this set contains a ball of radius $\delta$. Therefore $\eta'([a,b])$ contains a ball of radius $\delta$, concluding the proof.
\end{proof}

\section{Continuity of $\SLE_8$: proof of Theorem~\ref{thm:sle8_continuous}}
\label{sec:uniqueness_of_limit}

Suppose that $\nu > 0$, $\kappa' \in [4+\nu,8)$, $\kappa=16/\kappa'$, and $h$ is a whole-plane GFF with values modulo a global multiple of $2\pi \chi$. Let $z_+=10$. We consider four simple curves $\gamma^L_{z_+}$, $\gamma^R_{z_+}$, $\gamma^L_{0}$, and $\gamma^R_{0}$, such that $\gamma^L_{0}$ and $\gamma^R_{0}$ intersect and merge $\gamma^L_{z_+}$ and $\gamma^R_{z_+}$ at $10i$ and $-10i$, respectively, the pocket formed is contained in $B(0,20)$ and contains $B(5,4)$. Let $E$ be the event that the flow line~$\eta^L_{0}$ (resp.\ $\eta^R_{0}$) merges with $\eta^L_{z_+}$ (resp.\ $\eta^R_{z_+}$) inside $B(10i,1/100)$ (resp.\ $B(-10i,1/100)$) and all the flow lines travel at distance at most $1/100$ from the corresponding deterministic curves. By Lemmas~\ref{lem:uniform-far0}, \ref{lem:boundary-merge} there exists $p \in (0,1)$ depending only on $\nu$ so that
\begin{equation}
\label{eqn:p_e_lbd}
	\p[E] \geq p.
\end{equation}

Let $(\kappa_n')$ be a sequence in $[4+\nu,8)$ with $\kappa_n' \uparrow 8$, let $(\eta_n')$ be a sequence of whole-plane space-filling $\SLE_{\kappa_n'}$ from $\infty$ to $\infty$ normalized so that $\eta_n'(0) = 0$ and parameterized by Lebesgue measure and conditioned on the event $E$. By~\eqref{eqn:p_e_lbd}, Proposition~\ref{prop-sle-tight}, and the Arzel\`a-Ascoli theorem, there exists a weak subsequential limit of the laws of $(\eta_n')$ on the space of paths in $\C$ endowed with the local uniform distance. Let $\eta^L_{n,0}$, $\eta^R_{n,0}$, $\eta^L_{n,z_+}$, $\eta^R_{n,z_+}$ be the paths corresponding to $\eta^L_{0}$, $\eta^R_{0}$, $\eta^L_{z_+}$, $\eta^R_{z_+}$ for $\eta_n'$. Let $U_n$ be the component of $\C \setminus (\eta_{0,n}^L \cup \eta_{0,n}^R \cup \eta_{n,z_+}^L \cup \eta_{n,z_+}^R)$ which contains $5$, $x^L_n$ (resp.\ $x^R_n$) be the point on $\partial U$ where $\eta^L_{0}$ and $\eta^L_{z_+}$ (resp.\ $\eta^R_{0}$ and $\eta^R_{z_+}$) intersect each other, $S_n = \inf\{t \in \R : \eta_n'(t) \in U_n\}$, $T_n = \inf\{t \geq S_n : \eta_n'(t) \notin \closure{U_n}\}$, and let $\phi_n \colon \h \to U_n$ be the unique conformal map with $\phi_n(0) = \eta_n'(S_n)$, $\phi_n(\infty) = \eta_n'(T_n)$, and $\im(\phi_n^{-1}(5)) = 1$. We note that $S_n,T_n \in [0,20^2 \pi]$. For each $n \in \N$, we let $\wt{\eta}_n' = \phi_n^{-1}(\eta_n')$. Let $\wt{W}_t^n$ be the Loewner driving function for $\wt{\eta}_n'$ when parameterized by half-plane capacity. By passing to a further subsequence if necessary we can assume that we have the convergence of the joint law of $(\eta_n',\phi_n,S_n,T_n,\wt{W}_t^n)$ where we view $\phi_n$ as random variable taking values in the space of conformal maps defined on $\h$ equipped with the Carath\'eodory topology. Let $(\eta',\phi,S,T,\wt{W})$ have the law of the limit. By the Skorokhod representation theorem for weak convergence we may assume without loss of generality that $(\eta_n',\phi_n,S_n,T_n,\wt{W}^n)$ for $n \in \N$ and $(\eta',\phi,S,T,\wt{W})$ are coupled onto a common probability space so that we a.s.\ have that $\eta_n' \to \eta'$ locally uniformly, $\phi_n \to \phi$ in the Carath\'eodory topology, $S_n \to S$, $T_n \to T$, and $\wt{W}^n \to \wt{W}$ locally uniformly. We let $U$ be the range of $\phi$ and $\wt{\eta}'= \phi^{-1}(\eta')$.

Given a curve $\eta$ from $0$ to $\infty$ in $\h$ parameterized by half-plane capacity and for $r>0$, we let $\tau_r=\inf\{ t \geq 0 : \eta(t)\not\in B(0,r)\cap \closure{\h}  \}$. When parameterized by half-plane capacity, the curves $\wt{\eta}_n'$ are $\SLE_{\kappa_n'}(\kappa'_n/2-4;\kappa'_n/2-4)$ from $0$ to $\infty$ with force points at $\phi_n^{-1}(x^L_n)$ and $\phi_n^{-1}(x^R_n)$.  Since we are working on $E$, we note that there exists a constant $c>0$ such that $\phi_n^{-1}(x^L_n),\phi_n^{-1}(x^R_n)\not\in[-2c,2c]$ a.s. Let $\wh{\eta}_n'$ be a chordal $\SLE_{\kappa_n'}$ from $0$ to $\infty$ and let $M^n_t$ be the local martingale corresponding to $\wh{\eta}_n'$ as defined in~\cite[Theorem 6]{sw2005coordinate} with 2 in place of $n$ and with 
$z^n_1=\phi_n^{-1}(x^L_n)$, $z^n_2=\phi_n^{-1}(x^R_n)$, $\rho_1^n=\rho_2^n=\kappa'_n/2-4$. The law of $\wh{\eta}_n'$ weighted by $M^n_t$ coincides with the law of $\wt{\eta}_n'$ reparameterized by half-plane capacity, when both curves are stopped upon exiting $B(0,c)\cap \h$. Moreover we note that $\sup_{t\in[0,\tau_c]}|M^n_t|$  is bounded and a.s.\ tends to $1$ as $n \to \infty$.  Since the law of $\wh{\eta}_n'$ converges in the Carath\'eodory sense to the law of a chordal $\SLE_8$ from $0$ to $\infty$ as $n\to\infty$, we obtain that the law of~$\wt{\eta}_n'$ parameterized by half-plane capacity and stopped upon leaving $B(0,c)\cap\h$ also converges in the Carath\'eodory sense to the law of a chordal $\SLE_8$. Thus for $t\leq\tau_c$ we have $\wt{W}_t = \sqrt{8} B_t$ where $B$ is a standard Brownian motion.

 Let $\wt{A}_t^n$ (resp.\ $\wt{A}_t$) be the family of hulls associated with the chordal Loewner flow driven by $\wt{W}_t^n$ (resp.\ $\wt{W}_t$). Since $\eta'$ is a continuous curve and~$\phi$ is conformal hence continuous in~$\h$ (i.e., at least away from $\partial \h$) we have that $\wt{\eta}'|_{[s,t]}$ is continuous for any interval of time $[s,t]$ so that $\eta'([s,t]) \subseteq U$. This implies that $\wt{H}(r) = \hcap( \hull(\wt{\eta}'([0,r]))$ is continuous and by the monotonicity of the half-plane capacity $\wt{H}$ is also non-decreasing. We want to show that $\wt{H}$ is strictly increasing so that we can parameterize $\wt{\eta}'$ by capacity and obtain a curve which is continuous when it is away from~$\partial \h$. For each $t \geq 0$ we let 
\[ \wt{S}(t) = \inf\{ r \geq 0 : \wt{H}(r) = t\}.\]
By the continuity of $\wt{H}$, we have $\wt{H}(\wt{S}(t)) = t$. We also clearly have that $\wt{S}$ is strictly monotone and left-continuous. We want to show that it is continuous.

We will first show that for every $t \geq 0$ we have
\begin{equation}
\label{eqn:tilde_eta_contained}
\wt{\eta}'([0,\wt{S}(t)]) \cap\h \subseteq \wt{A}_t.
\end{equation}
Suppose that this does not hold for some $t > 0$. Then we can find $0<s_1<s_2< \wt{S}(t)$ and $\ep>0$ such that $\neigh{\epsilon}{\wt{\eta}'([s_1,s_2])} \subseteq \h \setminus \wt{A}_t$. This, in turn, would imply by the Carath\'eodory convergence of $(\wt{A}_t^n)$ to $(\wt{A}_t)$ that $\neigh{\epsilon/2}{\wt{\eta}'([s_1,s_2])}\subseteq \h \setminus \wt{A}_t^n$ for all $n \in \N$ large enough. Since the curves~$\wt{\eta}_n'$ converge locally uniformly to~$\wt{\eta}'$, at least in any interval of time in which~$\wt{\eta}'$ is away from $\partial \h$, we have $\wt{\eta}_n'([s_1,s_2])\subseteq \h \setminus \wt{A}_t^n$ for $n \in \N$ large enough. By continuity of the capacity w.r.t.\ Carath\'eodory convergence we have that
\[ \hcap( \hull(\wt{\eta}'([0,s_2])))=\lim_n\hcap( \hull(\wt{\eta}_n'([0,s_2])))\]
and thus $\hcap( \hull(\wt{\eta}_n'([0,s_2]))) < t$ for $n \in \N$ large enough, for otherwise 
\[ t=\hcap( \hull(\wt{\eta}'([0,\wt{S}(t)])))\geq\hcap( \hull(\wt{\eta}'([0,s_2])))=\lim_n\hcap( \hull(\wt{\eta}_n'([0,s_2])))\geq t\]
and $\hcap( \hull(\wt{\eta}'([0,s_2])))=t$ with $s_2<\wt{S}(t)$ would contradict the definition of $\wt{S}(t)$. The fact that $\hcap( \hull(\wt{\eta}_n'([0,s_2]))) < t$ for $n \in \N$ large enough contradicts having $\wt{\eta}_n'([s_1,s_2])\subseteq \h \setminus \wt{A}_t^n$ for $n \in \N$ large enough. Therefore~\eqref{eqn:tilde_eta_contained} holds. This in particular implies that $\hull(\wt{\eta}'([0,\wt{S}(t)]))\subseteq \wt{A}_t$. Since both of these hulls have half-plane capacity $t$, we can conclude
\begin{equation}
\label{eqn:tilde_eta_hull_equal}
\hull(\wt{\eta}'([0,\wt{S}(t)]))= \wt{A}_t \quad\text{for all}\quad t \geq 0.
\end{equation}

We next claim that, for $t\leq\tau_c$,
\begin{equation}
\label{eqn:tilde_a_before_t}
\wt{A}_t = \closure{\cup_{0 \leq h <t} \wt{A}_h}.
\end{equation}
Indeed, by~\cite[Theorem~3]{sw2005coordinate} a chordal $\SLE_8$ from $0$ to $\infty$ targeted at an interior point $y$ has the same law as a radial $\SLE_8(2)$ from $0$ to $y$ with force point at $\infty$, up to the first time that $y$ is disconnected from $\infty$. Moreover, $\SLE_8(2)$ processes do not disconnect their force points (as $\SLE_\kappa(\rho)$ processes do not when $\rho\geq \kappa/2-2$). By applying this to the points in $\h$ with rational coordinates, we thus see that a chordal $\SLE_8$ a.s.\ accumulates at every point in $\h$ before disconnecting it from $\infty$ which, together with the fact that $\wt{W}_t = \sqrt{8} B_t$ for $t\leq\tau_c$, proves the claim.

From this point up to the end of the proof we will restrict to times $t\in[0,\tau_c]$. Combining~\eqref{eqn:tilde_eta_hull_equal} and~\eqref{eqn:tilde_a_before_t} we have for every $t\in[0,\tau_c]$ that
 \begin{equation}
 \label{eqn:hull_equality}
 \hull(\wt{\eta}'([0,\wt{S}(t)])) = \wt{A}_t = \closure{\cup_{0 \leq h <t} \wt{A}_h} = \closure{\cup_{0 \leq h<t} \hull(\wt{\eta}'([0,\wt{S}(h)]))}.
 \end{equation}
We will now deduce from~\eqref{eqn:hull_equality} that $\hull(\wt{\eta}'([0,\wt{S}(t)]))=\wt{\eta}'([0,\wt{S}(t)])$ for every $t\in[0,\tau_c]$. To see this, we suppose for contradiction that there exists $t\in[0,\tau_c]$ and $x\in \hull(\wt{\eta}'([0,\wt{S}(t)]))\setminus \wt{\eta}'([0,\wt{S}(t)])$. For fixed $x$ we let $t_x$ be the infimum of the times such that this happens. As $\wt{A}_{t_x} = \cap_{s > t_x} \wt{A}_s$ we have that $x \in \wt{A}_{t_x} = \hull(\wt{\eta}'([0,\wt{S}(t_x)]))$ and we also have that $x \notin \wt{\eta}'([0,\wt{S}(t_x)])$. Thus $\wt{S}(t_x)$ is the disconnection time of $x$ from $\infty$. We note that $\partial \hull(\wt{\eta}'([0,\wt{S}(t_x)])) \subseteq \wt{\eta}'([0,\wt{S}(t_x)])$ and therefore $x$ is in the interior of $\hull(\wt{\eta}'([0,\wt{S}(t_x)]))$. Therefore there exists $\epsilon > 0$ such that $B(x,\epsilon)\subseteq \hull(\wt{\eta}'([0,\wt{S}(t_x)]))\setminus \wt{\eta}'([0,\wt{S}(t_x)])$. Fix $0 \leq h < t_x$. Then $\hull(\wt{\eta}'([0,\wt{S}(h)]))$ does not disconnect $x$ from $\infty$ and since $B(x,\epsilon)\cap \wt{\eta}'([0,\wt{S}(t_x)])=\emptyset$ we have $B(x,\epsilon)\cap \wt{\eta}'([0,\wt{S}(h)])=\emptyset$ as well. Thus $B(x,\epsilon)\subseteq \h \setminus \wt{A}_h$ for every $0 \leq h < t_x$ and $B(x,\epsilon)\not\subseteq \closure{\cup_{h<t_x}\wt{A}_h} $, giving the desired contradiction.

Note that $\hull(\wt{\eta}'([0,s]))=\hull(\wt{\eta}'([0,\wt{S}(\wt{H}(s))]))$, in fact, by the definition of $\wt{S}$ it holds $\wt{S}(\wt{H}(s))\leq s$, so $\hull(\wt{\eta}'([0,\wt{S}(\wt{H}(s))]))\subseteq\hull(\wt{\eta}'([0,s]))$, and the equality follows because both hulls have half-plane capacity $\wt{H}(s)$.   Let $\sigma_c=\inf\{s\geq 0:\wt\eta'(s)\not\in B(0,c)\cap \closure{\h}  \}$. We have that $\tau_c=\inf\{t\geq 0:\wt\eta'(\wt S(t))\not\in B(0,c)\cap \closure{\h}  \}=\wt{H}(\sigma_c)$. 
Indeed,  for all $t\geq0$ such that $\wt{\eta}'(\wt{S}(t))\not\in B(0,c)\cap \closure{\h}$ it holds $\wt{S}(t)\geq\sigma_c$ and $t=\wt{H}(\wt{S}(t))\geq \wt{H}(\sigma_c)$, hence $\tau_c\geq\wt{H}(\sigma_c)$.
 On the other hand, one can find $s>\sigma_c$ arbitrarily close to $\sigma_c$ such that   $\wt{\eta}'|_{[\sigma_c,s]}$ fills in a ball which is not entirely contained in  $B(0,c)\cap\closure{\h}$. Thus $\hull(\wt{\eta}'([0,\wt{S}(\wt{H}(s))]))=\hull(\wt{\eta}'([0,s]))\not\subseteq B(0,c)\cap\closure{\h}$, which implies $\tau_c\leq\wt{H}(s)$ and $\tau_c\leq\wt{H}(\sigma_c)$, by continuity of $\wt{H}$.
 Moreover we have $s\leq\sigma_c$ if and only if $ \wt{H}(s)\leq \wt{H}(\sigma_c)=\tau_c$. One implication is given by monotonicity of $\wt H$. For the other one, if  $s>\sigma_c$ then $\wt{H}(s)>\tau_c$ as $\wt{H}(s)=\tau_c$ cannot happen 
 as $\wt{\eta}'|_{[\tau_c,s]}$ fills in a ball which is not entirely contained in $B(0,c) \cap \h$.

Thus,
\[ \wt{\eta}'([0,s])\subseteq \hull(\wt{\eta}'([0,s]))=\hull(\wt{\eta}'([0,\wt{S}(\wt{H}(s))]))=\wt{\eta}'([0,\wt{S}(\wt{H}(s))])\subseteq \wt{\eta}'([0,s]) \]
which implies $\hull(\wt{\eta}'([0,s]))=\wt{\eta}'([0,s])$ and $\wt{S}(\wt{H}(s))=s$ (since in every interval of time $\eta'$ hence $\wt{\eta}'$ fills in a Euclidean ball we have that $\wt{\eta}'([0,s])$ is strictly increasing) for every all $s\in[0,\sigma_c]$. Therefore $\wt{H}|_{[0,\sigma_c]}$ cannot be locally constant, which implies that $\wt{S}|_{[0,\tau_c]}$ is continuous.

Suppose that $0 \leq s < t  \leq\tau_c$ is so that $\wt{\eta}' \circ \wt{S}$ is in $\h$ so that we know that the curve is continuous. Fix $R > 0$ large and assume further we are working on the event that $\wt{A}_t \subseteq B(0,R) \cap \h$. Fix $\epsilon > 0$. Then there exists $\delta > 0$ so that if $[u,v] \subseteq [s,t]$ with $|u-v| \leq \delta$ then $\diam(\wt{\eta}' \circ \wt{S}([u,v])) \leq \epsilon$. Therefore the probability that a Brownian motion starting from $2R i$ hits $\wt{\eta}' \circ \wt{S}([u,v])$ before exiting $\h$ is $O( (\log \epsilon^{-1})^{-1})$. This implies that the probability that a Brownian motion starting from $2Ri$ hits $\wt{g}_s(\wt{\eta}' \circ \wt{S}([u,v]))$ before exiting $\h$ is $O( (\log \epsilon^{-1})^{-1})$ which in turn implies that $\diam(\wt{g}_s(\wt{\eta}' \circ \wt{S}([u,v])))$ is $O( (\log \epsilon^{-1})^{-1})$. Therefore $\wt{g}_s \circ \wt{\eta}' \circ \wt{S}$ is continuous in $[s,t]$ on $\wt{A}_t \subseteq B(0,R) \cap \h$. Since $R > 0$ was arbitrary we conclude that $\wt{g}_s \circ \wt{\eta}' \circ \wt{S}$ is continuous in $[s,t]$. By the conformal Markov property for $\SLE_8$, this implies that $\SLE_8$ is generated by a continuous curve. \qed

\bibliographystyle{abbrv}
\bibliography{bibfile}

\end{document}